\documentclass[12pt]{amsart}
\usepackage{amsmath}
\usepackage{amsfonts}
\usepackage{amsthm}
\usepackage{amssymb}
\usepackage{amscd}
\usepackage[all]{xy}
\usepackage{enumerate}
\usepackage{hyperref}
\usepackage{stackrel}
\usepackage{graphicx}


\subjclass[2010]{}

\newcommand*{\ext}{\mathcal{E}\kern -.7pt xt}
\theoremstyle{plain}
\newtheorem{thm}{Theorem}[section]
\newtheorem{thml}{Theorem}

\newtheorem{prop}[thm]{Proposition}

\newtheorem{cor}[thm]{Corollary}

\newtheorem{lem}[thm]{Lemma}

\theoremstyle{definition}
\newtheorem{defn}[thm]{Definition}

\newtheorem{prob}[thm]{Problem}

\newtheorem{expl}[thm]{Example}

\newtheorem*{ackn}{Acknowledgment}

\newtheorem{rmk}[thm]{Remark}
\newcommand{\sA}{\mathcal{A}}
\newcommand{\sB}{\mathcal{B}}
\newcommand{\sC}{\mathcal{C}}

\newcommand{\sE}{\mathcal{E}}
\newcommand{\sF}{\mathcal{F}}
\newcommand{\sG}{\mathcal{G}}

\newcommand{\sI}{\mathcal{I}}

\newcommand{\sK}{\mathcal{K}}
\newcommand{\sL}{\mathcal{L}}

\newcommand{\sO}{\mathcal{O}}

\newcommand{\sQ}{\mathcal{Q}}

\newcommand{\sS}{\mathcal{S}}

\newcommand{\sU}{\mathcal{U}}

\newcommand{\sX}{\mathcal{X}}

\newcommand{\mC}{\mathbb{C}}
\newcommand{\mD}{\mathbb{D}}

\newcommand{\mL}{\mathbb{L}}

\newcommand{\mP}{\mathbb{P}}

\newcommand{\Kod}{\mathrm{Kod}\,}

\newcommand{\codim}{\mathrm{codim}\,}

\newcommand{\rank}{\mathrm{rank}\,}

\usepackage{color}

\numberwithin{equation}{section}

\newcommand{\beba}  {\begin{equation}\begin{array}{rcl}}

\newcommand{\eaee}  {\end{array}\end{equation}}

\makeatletter
\def\l@section{\@tocline{1}{0pt}{1pc}{}{}}
\def\l@subsection{\@tocline{2}{0pt}{1pc}{4.6em}{}}
\def\l@subsubsection{\@tocline{3}{0pt}{1pc}{7.6em}{}}
\renewcommand{\tocsection}[3]{%
  \indentlabel{\@ifnotempty{#2}{\makebox[2.3em][l]{%
    \ignorespaces#1 #2.\hfill}}}#3}
\renewcommand{\tocsubsection}[3]{%
  \indentlabel{\@ifnotempty{#2}{\hspace*{2.3em}\makebox[2.3em][l]{%
    \ignorespaces#1 #2.\hfill}}}#3}
\renewcommand{\tocsubsubsection}[3]{%
  \indentlabel{\@ifnotempty{#2}{\hspace*{4.6em}\makebox[3em][l]{%
    \ignorespaces#1 #2.\hfill}}}#3}
\makeatother

\title{Local systems, algebraic foliations, and fibrations}


\author{Luca Rizzi}
\address{Luca Rizzi\\ IBS Center for Complex Geometry, 55 EXPO-ro, Yuseong-gu, Daejeon, 34126, South Korea,
\texttt{lucarizzi@ibs.re.kr}}

\author{Francesco Zucconi}
\address{Francesco Zucconi\\Department of Mathematics, Computer Science and Physics \\
Universit\`a di Udine\\
Udine, 33100\\ Italia
\texttt{Francesco.Zucconi@dimi.uniud.it}}

\begin{document}

\markboth{}{}

\maketitle
\begin{abstract} Given a semistable fibration $f\colon X\to B$ we introduce a correspondence between foliations $\sF$ on $X$ and local systems $\mL$ on $B$. Building up on this correspondence we find conditions that give maximal rationally connected fibrations in terms of data on the foliation.
We prove the Castelnuovo-de Franchis theorem in the case of $p$-forms and we apply it to show when, under some natural conditions, a line subbundle of the sheaf of $p$-forms induces the Iitaka fibration.
%
\end{abstract}

\section{Introduction}
Let $f\colon X\to B$ be a (semistable) fibration between a smooth complex $n$-dimensional algebraic variety $X$ and a smooth curve $B$. Its relative tangent sheaf is a standard example of algebraically integrable foliation, and so are its algebraically integrable sub-foliations. On the other hand in \cite{RZ4} we have associated to $f\colon X\to B$ the local systems $\mD_{X}^k$ on $B$ given by the $k$-forms on the fibers of $f$ which are locally liftable to closed holomorphic $k$-forms on $X$,  $k=1,\cdots , n-1$. In this paper we merge these two dual approaches. Moreover we consider again $\mD_{X}^1$ but in light of the Castelnuovo-de Franchis theorem. 

%
\subsection{Foliations and local systems of $1$-forms}
The relative tangent sheaf $T_{X/B}$ is the foliation given by the kernel of the differential map $T_X\to f^*T_B$.
On the dual side, the local system $\mD_{X}^1$ is defined by the exact sequence
\begin{equation}
	\label{seqintro}
	0\to\omega_B\to f_{*}\Omega_{X,d}^1\to \mD_{X}^1\to 0
\end{equation}
where $\Omega_{X,d}^1$ is the sheaf of $d$-closed holomorphic $1$-forms on $X$: $\mD_{X}^1$ has been deeply studied in \cite{PT} if $n=2$. See also Section \ref{sez2} for the necessary background on foliations of vector fields and local systems of relative differential forms. 
\subsubsection{The correspondence}

In Section \ref{sez3} we show that each sub-foliation $\sF\subseteq T_{X/B}$  gives a local system $\mL_\sF\leq\mD_{X}^1$, and viceversa every local system $\mL\leq\mD_{X}^1$ gives a foliation $\sF_\mL\subseteq T_{X/B}$, hence we have a correspondence:
\begin{equation}
\label{corrispondeprima}
\big\{\text{foliations } \sF\subseteq T_{X/B}\big\}\stackrel[\beta]{\alpha}{\rightleftarrows} \big\{\text{local systems } \mL\leq \mD_{X}^1\big\}
\end{equation}
defined by $\alpha(\sF)=\mL_{\sF}$ and $\beta(\mL)=\sF_{\mL}$. In Proposition \ref{uguaglianzalocalsys} we show that in general these maps are not inverse of each other.

%

\subsubsection{The factorizing variety} In Section \ref{sez4} we solve the following problem. Let $Y$ be a normal $m$-dimensional variety which factors $f\colon X\to B$ as follows:
\begin{equation}
\xymatrix{
X\ar[r]^{h}&
Y\ar[r]^{g}&
B
}
\end{equation} 
where $g\circ h=f$ and $f,g$ are still fibrations. It is not difficult to see that in this case we have a sub local system $\mD_Y^1\leq\mD_{X}^1$ whose wedges satisfy some natural properties determined by the dimension of $Y$ and of the $g$-fibers. We ask to what extent we can obtain the viceversa in terms of local systems. More precisely, is it true that given a local system $\mL<\mD_X^1$ which mimic the properties of $\mD_Y^1$ then it gives back a variety $Y$ which factors $f\colon X\to Y\to B$? It turns out that this problem can be tackled in the framework presented in \cite{RZ4}.

The wedge product naturally gives maps $\bigwedge^{i} \mL\to  \mD_X^i$. The key definition is the following:
\begin{defn}
	We say that  $\mL$ is {\it{of Castelnuovo-type of order $m$}} if  $\rank\mL \geq m+1$, $\bigwedge^{m} \mL\to  \mD_X^m$ is zero while  $\bigwedge^{m-1} \mL\to  \mD_X^{m-1}$
	is injective on decomposable elements.
\end{defn}

\noindent This leads us to show:

\begin{thml}
	\label{thmA}
	Let $f\colon X\to B$ be a semistable fibration and $\mL\leq\mD_X$ a local system of Castelnuovo-type of order $m$. Then up to a covering $\widetilde{B}\to B$ and base change $\tilde{f}\colon \widetilde{X}\to \widetilde{B}$ there exist a normal $m$-dimensional complex space $\widetilde{Y}$ and holomorphic maps $\tilde{g}\colon\widetilde{Y}\to\widetilde{B}  $ and $\tilde{h}\colon \widetilde{X}\to\widetilde{B}$ such that $\tilde{f}=\tilde{h}\circ \tilde{g}$. Furthermore if  $\sF_{\mL}$ is the foliation associated to $\mL$ then it is algebraically integrable.
\end{thml}
\noindent

See Corollary \ref{castel1}. Theorem \ref{thmA} can also be generalized in the case of a flag of local systems $\mL_{s}<\dots<\mL_{2}<\mL_{1}<\mD_{X}^1$, see Corollary \ref{flag}.

\subsubsection{Local system versus algebraically integrable foliation}

On the other hand it is well-known that an algebraically integrable foliation $\sF$ contained in $T_{X/B}$ allows to recover an intermediate variety $Y$ up to birational equivalence; for example: see \cite[Lemma 4.12]{L} recalled in  Lemma \ref{lasic}.
 We show the following comparison theorem:

\begin{thml}
	\label{thmB}
	Let $f\colon X\to B$ be a semistable fibration and $\mL$ a Castelnuovo-type local system of order $m$. Then $\mL$ gives normal variety $Y'$ and complex spaces $\widetilde{Y}$, and $\widehat{Y}$, all of dimension $m$, which fit into the following diagram
	\begin{equation}
		\xymatrix{
			\widetilde{X}\ar[d]\ar[r]& X\ar[dd]& X'\ar[d]\ar[l]& \widehat{X}\ar[l]\ar[d]\ar@/_1.5pc/[lll]\\
			\widetilde{Y}\ar[d]& & Y'\ar[d]& \widehat{Y}\ar[d]\ar@{-->}[l]\ar@{-->}@/_1.5pc/[lll]\\
			\widetilde{B}\ar[r]& B& B\ar@{=}[l]& \widetilde{B}\ar[l]\ar@{=}@/^1.3pc/[lll]
		}
	\end{equation}
	\newline
	where $\widetilde{B}\to B$ is a covering, $\widehat{X}\to \widetilde{X}$ is a (generically) degree 1 map, $\widehat{Y}\dashrightarrow Y'$ is a dominant meromorphic map and $\widehat{Y}\dashrightarrow \widetilde{Y}$ is a bimeromorphic map.
\end{thml}

\noindent See Theorem \ref{dxsx}. This means that a local system of Castelnuovo-type recovers the intermediate variety either up to (possibly non finite) \'etale  cover or up to birational morphism. In Theorem \ref{thmA} and \ref{thmB} the algebricity of $\widetilde{X}, \widehat{X},\widetilde{Y},\widehat{Y}$ highly depends on specific assumptions. In any case Theorem \ref{thmB} should be read in the light of the theory of towers of varieties; \cite{P}. Finally the notion of local system of Castelnuovo-type gives us a nice theorem to study some of the fixed points of the above correspondence, that is to find conditions on a local system
$\mL\leq \mD_X$ or on an algebraically integrable foliation $\sF\subseteq T_{X/B}$ such that $\alpha(\beta(\mL))=\mL$ or $\beta(\alpha(\sF))=\sF$; see: Theorem \ref{fissi}.

\subsection{Maximal Rationally Connected fibrations}
Our point of view can give an answer to some problem concerning rational connected varieties.
First we recall that once we fix an ample polarization $\alpha$ on $X$ it defines the $\alpha$-slope $\mu(\sQ)$ of any torsion-free coherent sheaves on $X$ as: $\mu(\sQ):=\frac{\deg\sQ}{\rank \sQ}$. Let 
$$
0= \sF_0 \subsetneq \sF_1 \subsetneq\dots\subsetneq \sF_r =T_{X/B}
$$ 
be the $\alpha$-Harder-Narasimhan filtration of the relative tangent sheaf $T_{X/B}$ and we call $\sQ_i$ the quotients $\sF_i/\sF_{i-1}$.  We set
$$
i_{\max}:=\max\{0<i<k\mid \mu(\sQ_i)>0\}\cup \{0\}.
$$
If $i_{\max} >0$, it is well known that for every index 
$0< i\leq i_{\max} $, $\sF_i$ is an algebraically integrable foliation with rationally connected general leaf, see \cite[Proposition 3.8, Corollary 3.10]{K}. In particular for such $i$ we obtain rational maps $X\dashrightarrow W_i$ with rationally connected general fibers.
We recall in Definition \ref{definitionMRC} the definition of an MRC-fibration, see also \cite[Sect. IV.5]{Kol}, \cite[Sect. 5]{D} and \cite[Thm. IV.5.5]{Kol} for the functoriality of the MRC-fibration. See also \cite{SCT}. In \cite{K}, Question 3.12 poses the problem of the characterization of  MRC-fibrations. In other words it is an open problem to understand when the rational map $X\dashrightarrow W_{i_{\max}}$ is an MRC-fibration over the base $B$. The study of $\sF_{i_{\max}}$ together with the local system $\mD_{X}^1$ allows us to give a solution of this problem, again in the case when this local system is of Castelnuovo-type, see Theorem \ref{mrc}.

\begin{thml}
	\label{thmC}
	Let $f\colon X\to B$ be a semistable fibration and  
	$$
	0= \sF_0 \subsetneq \sF_1 \subsetneq\dots\subsetneq \sF_r =T_{X/B}
	$$ the $\alpha$-polarized Harder-Narasimhan filtration of $T_{X/B}$. Assume that $\mD_{X}^1$ is of Castelnuovo-type of order $m$.
	If the foliation $\sF_{i_{\max}}$ has generic rank $n-m$, then $X\dashrightarrow W_{i_{\max}}$ is the relative MRC fibration of $f$. 
\end{thml}

\subsection{Generalisation of the Castelnuovo-de Franchis theorem}

The Castelnuovo-de Franchis theorem plays an important role in algebraic geometry. This theorem, together its  generalisations, essentially concerns 1-forms; in Section \ref{sez6} we show an extension to p-forms.

Let $\omega_1,\dots, \omega_l \in H^0 (X,\Omega^p_X)$, $l\geq p+1$, be linearly independent $p$-forms such that $\omega_i\wedge\omega_j=0$ (as an element of $\bigwedge^2\Omega^p_X$ and not of $\Omega_X^{2p}$) for any choice of $i,j=1,\dots, l$. These forms generate a subsheaf of $\Omega^p_X$ generically of rank $1$. We denote it by $\sL$. Note that the quotients $\omega_i/\omega_j$ define a non-trivial global meromorphic function on $X$ for every $i\neq j$, $i,j=1,\dots, l$. By taking the differential $d (\omega_i/\omega_j)$ we then get global meromorphic $1$-forms on $X$. We ask that there exist $p$ of these meromorphic differential forms $d (\omega_i/\omega_j)$ that do not wedge to zero; if this is the case we call  the subset $\{\omega_1,\dots, \omega_l \}\subset H^0 (X,\Omega^p_X)$ $p$-strict, see Definition \ref{definizionestrict}. For this new setting, this condition is  analogous to the strictness condition considered in \cite[Definition 2.1 and 2.2]{Ca2}, see also \cite[Definition 4.4]{RZ4}.
\begin{thml}
	\label{thmD}
	Let $X$ be an $n$-dimensional smooth variety and let $\{\omega_1,\dots, \omega_l \}\subset H^0 (X,\Omega^p_X)$ be a $p$-strict set. Then there exists a rational map $f\colon X\dashrightarrow Y$ over a $p$-dimensional smooth variety $Y$ such that the $\omega_i$ are pullback of some meromorphic $p$-forms $\eta_i$ on $Y$, $\omega_i=f^*\eta_i$, where $i=1,\dots , l$.
\end{thml}
\noindent See Theorem \ref{cast2}.

If $p=1$ then the rational map $f\colon X\dashrightarrow Y$ turns out to be a morphism, if we allow $Y$ to be normal. If $p\geq 2$ this is not always the case.
\subsection{Applications of the generalised Castelnuovo-de Franchis theorem: Iitaka fibrations}
We can apply Theorem \ref{thmD} to the case of Iitaka fibrations arising from subbundles $\sL$ of $\Omega^p_X$. 

We recall that if $X$ is a smooth projective variety of Kodaira dimension $\Kod X\geq 0$, then by a well-known construction of Iitaka, see cf. \cite{F}, \cite{Laz}, there is a birational morphism $u_\infty \colon X_\infty \to X$ from a smooth projective variety $X_\infty$, and a contraction $\phi_\infty\colon X_\infty\to Y_\infty$
onto a projective variety $Y_\infty$ such that a (very) general fiber $F$ of $\phi_\infty\colon X_\infty\to Y_\infty$ is smooth with Kodaira dimension zero, and $\dim Y_\infty$ is equal to $\Kod X$. The map $\phi_\infty\colon X_\infty\to Y_\infty$ is unique up to birational equivalence and it is referred to as an Iitaka fibration of $X$. In the case of a line bundle $\sL=\sO_X(L)$ such that $\bigoplus H^0(X,nL)$ is a finitely generated $\mathbb C$-algebra, the dimension of ${\rm{Proj}} \bigoplus_{n\in\mathbb N}H^0(X,nL)$ is called Kodaira dimension of $\sL$ and it is denoted by $\Kod \sL$.

\begin{thml}
	\label{thmE} Let $p\in\mathbb N$ and $1\leq p\leq n$. Let $X$ be a smooth variety of dimension $n$.
	If $\sL\hookrightarrow \Omega^p_X$ is an invertible subsheaf which is globally generated by a $p$-strict subset and if $\Kod X=\Kod \sL=p$, then the Stein factorization of $\varphi_{|\sL|}\colon X\to \mP(H^0(X,\sL)^\vee)$ induces the Iitaka fibration.
\end{thml}
See Theorem \ref{Iitakauno}.

Finally we present another application of Theorem \ref{thmD}. We stress that this application is only conjectural. Indeed it asserts that the same conclusion of Theorem \ref{thmE} can be obtained even in the case where $\sL$ is not globally generated but assuming that there exists $a,b\in \mathbb N$ such that $aK_X-bL$ is nef where $\sL=\sO_X(L)$; see Theorem  \ref{Iitakadue}. We maintain it in this paper because it can really be considered as an evidence for abundance conjecture.

\begin{ackn}
	The first author has been supported by JSPS-Japan Society for the Promotion of Science (Postdoctoral Research
	Fellowship, The University of Tokyo) and the IBS Center for Complex Geometry. The second author has been supported by the grant DIMA Geometry PRIDZUCC and by PRIN 2017 Prot. 2017JTLHJR \lq\lq Geometric, algebraic and analytic methods in arithmetics\rq\rq.
\end{ackn}

\section{Setting: fibrations, foliations, local systems}
\label{sez2}
Let $X$ be a smooth complex $n$-dimensional variety and $B$ a smooth complex curve.
In this paper we consider semistable fibrations $f\colon X\to B$; we denote by $X_b=f^{-1}(b)$ the fiber over a point $b\in B$ and assume that all the fibers $X_b$ are either smooth or reduced and normal crossing divisors. We recall that $\omega_{X/B}:=f^*\omega_B^\vee\otimes \omega_X$ is the relative dualizing sheaf and $\Omega^1_{X/B}$ is the sheaf of relative differentials defined by the short exact sequence 
\begin{equation}
\label{relativo}
0\to f^*\omega_B\to \Omega^1_X\to \Omega^1_{X/B}\to 0.
\end{equation}
We also recall that in this setting $\omega_{X/B}$ is locally free while $\Omega^1_{X/B}$ and its wedges $\Omega^k_{X/B}=\bigwedge^k \Omega^1_{X/B}$ are in principle only torsion free, see: c.f. \cite[Section 2]{RZ4}. The direct images $f_*\Omega^k_{X/B}$ are torsion free on the curve $B$ and hence also locally free.

\subsection{Foliations}
Let's start by recalling the definition of foliation that we will use in this paper.
\begin{defn}
A foliation is a saturated
subsheaf $\sF\subseteq T_X$ which is closed under the Lie bracket, i.e. $[\sF,\sF]\subseteq \sF$. The singularity
locus of a foliation is the subset of $X$ on which $\sF$ is not locally free, and
it has codimension at least 2. A leaf of $\sF$ is the maximal connected, locally
closed submanifold $L$ such that $T_L =\sF|_L$.

We also recall that a foliation $\sF$ is called algebraically integrable if its leaves are algebraic. 
\end{defn}

Of course from the fibration $f\colon X\to B$ we have the foliation induced by the relative tangent sheaf, that is the kernel of the differential map. This kernel is usually denoted by $T_{X/B}$ and it fits in  the following exact sequence, dual of (\ref{relativo}),
\begin{equation}
\label{tanrel}
0\to T_{X/B}\to T_X\to f^*T_B\to N\to 0
\end{equation} where $N=\ext^1(\Omega^1_{X/Y}, \sO_X)$ is a torsion sheaf supported on the critical locus of $f$. In this case we say that the foliation is induced by the fibration.
Of course foliations induced by fibrations are algebraically integrable. 

Algebraically integrable foliations give in some sense a viceversa of this construction by the following result, see for example \cite[Lemma 4.12]{L}
\begin{lem}
\label{lasic}
Let $X$ be a smooth projective variety and let $\sF$ be an algebraically integrable
foliation on $X$. Then there is a unique irreducible closed subvariety $W$ of $\textnormal{Chow}(X)$ whose general point parametrizes the closure of a general leaf of $\sF$. In other words, if $U \subseteq W \times X$ is
the universal cycle with projections $\pi\colon U\to W$ and $e \colon U \to X$, then $e$ is birational and $e(\pi^{-1})(w)\subseteq X$ is the closure of a leaf of $\sF$ for a general point $w\in W$.
\begin{equation}
\xymatrix{
U\ar[d]_{e}\ar[r]^{\pi}&W\\
X&
}
\end{equation}  
Then there exists a foliation $\widehat{\sF}$ on the normalisation $\nu\colon U^\nu\to U$ induced by $\pi\circ\nu$ and
which coincides with $\sF$ on $(e\circ \nu)^{-1}(X^\circ)$, where $X^\circ$ is a big open subset of $X$.
\end{lem}

We also recall the definition of algebraic and transcendental part of a foliation.
Let $\sF$ be a foliation on $X$. There exist a normal variety $Y$, unique up to birational equivalence,
a dominant rational map with connected fibers $\varphi\colon X\dasharrow Y$, and a foliation $\sG$ on $Y$ such that the following
conditions hold.
\begin{enumerate}
\item $\sG$ is purely transcendental, i.e., there is no positive-dimensional algebraic subvariety through a general point of $Y$ that is tangent to $\sG$.
\item $\sF$ is the pull-back of $\sG$ via $\varphi$. This means the following. Let $X^\circ\subset X$ and $Y^\circ\subset Y$ be smooth open subsets such that $\varphi$ restricts to a smooth morphism $\varphi^\circ$. Then $\sF|_{X^\circ}=(d\varphi^\circ)^{-1}\sG|_{Y^\circ}$.
\end{enumerate}
\begin{defn}
\label{algfol}
The foliation $\sF^a$ induced by $\varphi$ is called the algebraic part of $\sF$ while $\sG$ is its transcendental part.
\end{defn}
See for example \cite{AD}.

\subsection{Local systems}
The local systems on $B$ associated to $f$ are defined as follows. Consider again Sequence (\ref{relativo}) and its wedges
\begin{equation}
0\to f^*\omega_B\otimes \Omega^{k-1}_{X/B}\to \Omega^k_X\to \Omega^{k}_{X/B}\to 0
\end{equation}  for $k=1,\dots,n-1$.
By pushforward we can write
\begin{equation}
\label{seqhodge}
0\to \omega_B\otimes f_*\Omega^{k-1}_{X/B}\to f_*\Omega^k_X\to f_*\Omega^{k}_{X/B}\to R^1f_*\Omega^{k-1}_{X/B}\otimes \omega_B\to \dots.
\end{equation} and we take the corresponding sub-sequence of de Rham closed holomorphic forms as follows 
\begin{equation}
\label{diag}
\xymatrix{
0\ar[r] &\omega_B\otimes f_*\Omega_{X/B}^{k-1}\ar[r]\ar@{=}[d]& f_*\Omega_{X,d}^{k}\ar[r]\ar@{^{(}->}[d]& f_*\Omega^{k}_{X/B,d_{X/B}}\ar[r]\ar@{^{(}->}[d]& \dots\\
0\ar[r] &\omega_B\otimes f_*\Omega_{X/B}^{k-1}\ar[r]& f_*\Omega_{X}^{k}\ar[r]& f_*\Omega^{k}_{X/B}\ar[r]& \dots\\
}
\end{equation}
This gives the following definition:
\begin{defn}
\label{locsys}
We call $\mD^k_X$, for $k=1,\dots,n-1$, the image of the map $f_*\Omega_{X,d}^{k}\to f_*\Omega^{k}_{X/B,d_{X/B}}$.
\end{defn}
The $\mD^k_X$ are indeed local systems. In \cite{RZ4} we have proved this result for $\mD^1_X$ and $\mD^{n-1}_X$. The proof for $2\leq k\leq n-2$ is similar, hence it will be omitted here
\begin{prop}
 $\mD^k_X$ is a local system on $B$ for $k=1,\dots,n-1$.
\end{prop}
\begin{proof}
See \cite[Lemma 3.4 and 3.6]{RZ4}.
\end{proof}

Recall that by the famous Fujita's decomposition theorem, see \cite{Fu}, \cite{Fu2}, it holds that:
$$
f_*\omega_{X/B}=\sU\oplus\sA
$$ 
where $\sU$ is a unitary flat vector bundle and $\sA$ is an ample one. By the correspondence between unitary flat vector bundles and local systems, cf. \cite{De}, Fujita's decomposition gives naturally a local system $\mathbb U$ on $B$. There is a vast literature on this topic; see cf. \cite{BZ1}, \cite{BZ2}, \cite{CD1}, \cite{CD2}, \cite{CD3}, \cite{CK}. 

In \cite{RZ4} we have shed some light on the higher dimensional geometry associated to $\mathbb U$ and $\mD_{X}^1$ and their respective monodromies. In particular we have shown that $\mD_{X}^{n-1}=\mathbb U$ is the local system of relative top forms on the fibers.

In this paper we are mostly concerned with $\mD^1_X$. We denote $\mD_X:=\mD^1_X$ for simplicity, and we point out that, by Definition \ref{locsys}, it fits into the following short exact sequence 
\begin{equation}
\label{dx}
0\to \omega_B\to f_*\Omega^1_{X,d}\to \mD_X\to 0.
\end{equation}
\section{Natural correspondence between foliations and local systems}
\label{sez3}
The above discussion shows that in the case of a fibration $f\colon X\to B$, we can define an associated foliation and a local system. In this section we define a precise correspondence between these objects.

\subsection{Local systems defined by foliations}
Take $\sF\subseteq T_{X/B}\subset T_X$ a foliation on $X$ and consider the exact sequence of its inclusion in the tangent sheaf
\begin{equation}
0\to \sF\to T_X\to \sK\to 0
\end{equation} where $\sK$ is torsion free because by definition $\sF$ is saturated. Actually since $X$ is smooth and hence $T_X$ is locally free, $\sF$ is reflexive.

By taking the dual we obtain an exact sequence 
\begin{equation}
	\label{uno}
0\to \sK^\vee\to \Omega_X^1\to \sF^\vee\to M\to 0 
\end{equation} where $M$ is supported on the singular locus of the singular fibers. We call $\sQ$ the cokernel 
\begin{equation}
0\to \sK^\vee\to \Omega_X^1\to \sQ\to 0 
\end{equation}
and pushing forward via $f_*$ we have
\begin{equation}
0\to f_*\sK^\vee\to f_*\Omega_X^1\to f_*\sQ\to \dots 
\end{equation} Now it is straightforward to see from the condition $\sF\subseteq T_{X/B}$ that the inclusion of vector bundles $\omega_B\subset f_*\Omega_X^1$ factors through $f_*\sK^\vee$
\begin{equation}
\xymatrix{
&\omega_B\ar@{^{(}->}[d]\ar@{^{(}->}[rd]&&&\\
0\ar[r]&f_*\sK^\vee\ar[r]&f_*\Omega_X^1\ar[r]&f_*\sQ^\vee\ar[r]&\dots
}
\end{equation} Restricting this diagram to de Rham closed forms $f_*\Omega_{X,d}^1$ we obtain the commutative triangle
\begin{equation}
\xymatrix{
\omega_B\ar@{^{(}->}[d]\ar@{^{(}->}[rd]&\\
f_*\sK_d^\vee\ar[r]&f_*\Omega_{X,d}^1
}
\end{equation} where $f_*\sK_d^\vee$ indicates the intersection of $f_*\sK^\vee$ with the sheaf of de Rham closed forms. The cokernel of the diagonal arrow is, by Sequence (\ref{dx}), the local system $\mD_X$, hence the cokernel of $\omega_B\hookrightarrow f_*\sK_d^\vee$ is a local subsystem of $\mD_X$ which we denote by $\mL_\sF$, and we call the local system obtained from the foliation $\sF$. By definition the following sequence is exact:
\begin{equation}
	\label{kappa}
	0\to \omega_B\to f_*\sK_d^\vee\to \mL_\sF\to 0.
\end{equation}

Actually $\mD_X/ \mL_\sF$ is also a local system and we have an exact sequence
\begin{equation}
	\label{esattalocalsys}
0\to \mL_\sF\to \mD_X\to \mD_X/ \mL_\sF\to 0.
\end{equation} 

In this paper we will work mainly with the local system $\mL_\sF$, but of course one could equivalently consider the cokernel $\mD_X/ \mL_\sF$.

\begin{expl}
	\label{casibanali}
	We look at the two extreme cases $\sF=0$ and $\sF=T_{X/B}$. 
	
	For $\sF=0$, we have that $\mL_\sF=\mD_{X}$  and Sequence \ref{esattalocalsys} is 
	$$
	0\to \mD_{X}\to \mD_{X}\to 0\to 0.
	$$
	
On the other hand if $\sF=T_{X/B}$ we have $\mL_\sF=0$. In fact by  (\ref{tanrel}) we have the exact sequence 
$$
0\to \sK\to f^*T_B\to N\to 0.
$$  Dualizing we have the inclusion $f^*\omega_B\to \sK^\vee$ which fits into the diagram 
\begin{equation}
\xymatrix{
0\ar[r]&f^*\omega_B\ar@{^{(}->}[d]\ar[r]&\Omega_X^1\ar@{=}[d]\ar[r]&\Omega^1_{X/B}\ar[d]\ar[r]&0\\
0\ar[r]&\sK^\vee\ar[r]&\Omega_X^1\ar[r]&\sQ^\vee\ar[r]&0
}
\end{equation} 
By the commutativity of the diagram the co-kernel of the injection $f^*\omega_B\to \sK^\vee$ is also the kernel of $\Omega^1_{X/B}\to \sQ^\vee$ hence it is zero because it should be a torsion sheaf inside the torsion free sheaf $\Omega^1_{X/B}$ (recall that $f$ is semistable). This means that $f^*\omega_B\cong\sK^\vee$ and we easily have the result. Sequence \ref{esattalocalsys} is 
	$$
0\to 0\to \mD_{X}\to \mD_{X}\to 0.
$$
\end{expl}

\begin{rmk}
Note that so far we have not made any particular assumptions on the foliation $\sF$; we show now that in this framework it is not restrictive to consider algebraically integrable foliations.

In fact take $\sF$ an arbitrary foliation and call $Y$ as in Definition \ref{algfol} where the dominant rational map $\varphi\colon X\dashrightarrow Y$ defines the algebraic and the transcendental part of $\sF$.  In our setting we also have a morphism $Y\to B$ and,
since $\sF^a\subseteq \sF\subseteq T_{X/B}\subset T_X$, we can associate a local system $\mL_{\sF^a}$ to the algebraic part $\sF^a$.
It is not difficult to see that the inclusion $\sF^a\subseteq \sF$ is reversed at the level of local systems as $\mL_{\sF}\leq \mL_{\sF^a}$. This construction shows that, even in the case of a general foliation, we can always consider the local system associated to its algebraic part and reduce to the algebraic case. Therefore in this paper we will be mostly concerned with algebraically integrable foliations.
\end{rmk}

\subsubsection{The extrinsic construction of the local system $\mL_\sF$}
The above construction of $\mL_{\sF}$ is fundamentally intrinsic. We present now a more extrinsic alternative construction, which will be very useful in the following. 

Take $\sF\subseteq T_{X/B}\subset T_X$ an algebraically integrable foliation. The following diagram
\begin{equation}
\label{diagfol}
\xymatrix{
U\ar[d]_{e}\ar@/_2pc/[dd]_{\tilde{f}}\ar[r]^{\pi}&W\ar[ddl]\\
X\ar[d]_{f}&\\
B
}
\end{equation} 
follows by the one of Lemma \ref{lasic} in the relative setting. Possibly by blowing up we assume that $U$ and $W$ are smooth; hence note that $W$ is no longer necessarily in $\textnormal{Chow}(X)$. Call $\tilde{f}$ the composition $f\circ e$ and  $\tilde{\sF}=T_{U/W}$. Note that $\tilde{\sF}\subseteq T_{U/B}$ hence $\tilde{f}$ factors through $W$.
Now consider the sequence of the relative tangent sheaf of $\pi$:
\begin{equation}
0\to \tilde{\sF}\to T_U\to \pi^*T_W\to N\to 0.
\end{equation} We call $\sL$ the image of $T_U\to \pi^*T_W$
\begin{equation}
	\label{due}
0\to \tilde{\sF}\to T_U\to \sL\to 0
\end{equation}
Taking the dual of $T_U\to \sL$ and the direct image via $\tilde{f}_*$ gives  
\begin{equation}
\tilde{f}_*\sL^\vee\to \tilde{f}_*\Omega^1_U
\end{equation} 
Now since $e$ is a birational morphism, we have that $\tilde{f}_*\Omega^1_U\cong {f}_*\Omega^1_X$, hence taking the closed forms as before we obtain 

\begin{equation}
	\label{ll}
\xymatrix{
\omega_B\ar@{^{(}->}[d]\ar@{^{(}->}[rd]&\\
\tilde{f}_*\sL_d^\vee\ar[r]&f_*\Omega_{X,d}^1\cong \tilde{f}_*\Omega^1_{U,d}.
}
\end{equation}
and the cokernel of the vertical map gives a sublocal system of $\mD_X$. By analogy we call $\mD_W$ this local system but we immediately prove that this construction agrees with the one seen before. We state the following easy Lemma for later reference.

\begin{lem}
\label{lemma}
If we have a commutative diagram of sheaves as follows
\begin{equation}
\xymatrix{
	0\ar[r]&\sA\ar[r]&\sB\ar[r]&\sC \ar[r]&0\\
	0\ar[r]&\sA'\ar[r]&\sB\ar[r]\ar@{=}[u]&\sC'\ar@{^{(}->}[u]\ar[r]&0
}
\end{equation} then $\sA\cong \sA'$, hence also $\sC\cong \sC'$.
\end{lem}
\begin{proof}
It is immediate that we have an injective morphism $\sA'\hookrightarrow \sA$. The cokernel of this map is by commutativity isomorphic to the kernel of $\sC'\hookrightarrow \sC$, hence zero and we have proved the desired result.
\end{proof}
\begin{prop}
	\label{stesso}
The local system $\mD_W$ coincides with $\mL_\sF$.
\end{prop}
\begin{proof}
	We compare the two constructions as follows. Recall that $\mL_{\sF}$ is defined by taking the pushforward via $f$ of the exact Sequence (\ref{uno})
	\begin{equation}
		0\to \sK^\vee\to \Omega_X^1\to \sF^\vee\to M\to 0 
	\end{equation} and considering the quotient of ${f}_*\sK^\vee_d$ with kernel $\omega_B$ as in (\ref{kappa}). On the other hand for $\mD_W$ we dualize Sequence (\ref{due})  to obtain
\begin{equation}\label{primaoccorrenza}
	0\to \sL^\vee\to \Omega^1_{U}\to \tilde{\sF}^\vee\to M'\to 0
\end{equation}  and we proceed exactly as before after taking the pushforward via $\tilde{f}$.
On $B$ we can compare the two sequences obtained after taking the direct image
\begin{equation}
\xymatrix{
	0\ar[r]&f_*\sK^\vee\ar[r]&f_*\Omega^1_{X}\ar[r]&f_*\sF^\vee\ar[r]&\dots\\
	0\ar[r]&\tilde{f}_*\sL^\vee\ar[r]&\tilde{f}_*\Omega^1_{U}\ar[r]\ar@{=}[u]&\tilde{f}_*\tilde{\sF}^\vee\ar[r]&\dots
}
\end{equation}
The equality comes from the fact that $e$ is birational and $\Omega^1_X$ and $e_*\Omega^1_U$ coincide on a open subset of $X$ with complement of codimension at least 2, see also \cite[Exercise 5.3 page 419]{H} in the case of surfaces. 

Similarly we also have a map $\tilde{f}_*\tilde{\sF}^\vee\to f_*\sF^\vee$ as follows. If we consider $A\subset B$ an open subset, every section in $\Gamma(A,\tilde{f}_*\tilde{\sF}^\vee)$ is a section in  $\Gamma(\tilde{f}^{-1}(A),\tilde{\sF}^\vee)$. Now $\tilde{\sF}^\vee$ and $\sF^\vee$ coincide on an open subset in $U$ of the form $e^{-1}(X^0)$, where $X^0$ is an open subset with complement of codimension at least 2 in $X$. Hence our section gives a section of $\sF^\vee$ on the intersection $X^0\cap f^{-1}(A)$. Since $\sF^\vee$ is reflexive, this gives a section of $\Gamma(f^{-1}(A),{\sF}^\vee)=\Gamma(A,{f}_*{\sF}^\vee)$ by the Hartogs principle. This map is injective hence by Lemma \ref{lemma} the above diagram can be completed as follows 
\begin{equation}
	\xymatrix{
		0\ar[r]&f_*\sK^\vee\ar[r]&f_*\Omega^1_{X}\ar[r]&f_*\sF^\vee\ar[r]&\dots\\
		0\ar[r]&\tilde{f}_*\sL^\vee\ar[r]\ar@{=}[u]&\tilde{f}_*\Omega^1_{U}\ar[r]\ar@{=}[u]&\tilde{f}_*\tilde{\sF}^\vee\ar[r]\ar[u]&\dots
	}
\end{equation}
The identity $f_*\sK^\vee=\tilde{f}_*\sL^\vee$ immediately gives the thesis by taking the de Rham closed forms and the quotient of kernel $\omega_B$. 
\end{proof}
  \subsection{Foliations defined by local systems}
  \label{localtofoliation}
We can reverse the point of view to obtain a foliation $\sF\subseteq T_{X/B}$ starting from a local system $\mL\leq\mD_{X}$. By the exact Sequence \ref{dx} 
we obtain the following diagram 
\begin{equation}
	\label{ss}
	\xymatrix{
		0\ar[r]&\omega_B\ar[r]&f_*\Omega^1_{X,d}\ar[r]&\mD_{X}\ar[r]&0\\
		0\ar[r]&\omega_B\ar[r]\ar@{=}[u]&\sS\ar[r]\ar@{^{(}->}[u]&\mL\ar[r]\ar@{^{(}->}[u]&0
	}
\end{equation}
which defines $\sS$ as a subsheaf of $f_*\Omega^1_{X,d}$.

From the natural map $f^*f_*\Omega_{X}^1\to \Omega_{X}^1$ we then get the map $\eta\colon f^{-1}\sS\to \Omega_{X,d}^1\hookrightarrow\Omega_{X}^1$. The image sheaf $\sS_f$ is a subsheaf of $\Omega_{X}^1$. 
Taking the saturation of the subsheaf of $T_{X/B}$ given by the vector fields vanishing on $\sS_f$ we get a foliation $\sF_{\mL}$. More precisely call $\sF'$
\begin{equation}
	\label{defF}
\sF'(U):=\{v\in T_{X/B}(U)\mid \forall x\in U, \exists U_x\text{ with }x\in U_x\subset U\text{ such that } \iota_v s=0 \text{ for every }s\in \sS_f(U_x)   \}
\end{equation} where as usual $\iota$ denotes the contraction. Note that actually $\sF'$ is a sheaf since it inherits its sheaf properties from $T_{X/B}$.
We show that it is closed under Lie bracket. Let $v,w$ be sections of $\sF'$, then
$$
\iota_{[v,w]} s = \mathcal{L}_v \iota_w s- \iota_w \mathcal{L}_v s=- \iota_w \mathcal{L}_v s
$$ where $\sL$ is the Lie derivative. Since also 
$$
\mathcal{L}_v s=\iota_v ds+d(\iota_v s)=0
$$ we have the desired result since $s\in \sS_f\subset \Omega^1_{X,d}$ is closed. The saturation $\sF_{\mL}$ of $\sF'$ is then a foliation on $X$. Indeed $\sF_{\mL}$ and $\sF'$ coincide on an open dense subset $X^0$ of $X$. Hence if we take $v\in \Gamma(\sF_{\mL},V)$ a section of $\sF_{\mL}$ on an open subset $V$, then the map given by the Lie bracket
$$
[v,-]\colon \sF_{\mL}|_V\to T_X/\sF_{\mL}|_V
$$ is zero on $X^0\cap V$. By the very definition of saturation, $T_X/\sF_{\mL}$ is torsion free, this implies that the above map is identically zero, hence $\sF_{\mL}$ is also closed under Lie bracket. 
Finally we note that  $\sF_{\mL}$ is not necessarily an algebraically integrable foliation.

\subsection{The Correspondence}

By the above constructions we have two maps between the set of foliations contained in $T_{X/B}$ and the set of local systems contained in $\mD_{X}$
\begin{equation}
\label{corrisponde}
\big\{\text{foliations } \sF\subseteq T_{X/B}\big\}\stackrel[\beta]{\alpha}{\rightleftarrows} \big\{\text{local systems } \mL\leq \mD_{X}\big\}
\end{equation}
defined by $\alpha(\sF)=\mL_{\sF}$ and $\beta(\mL)=\sF_{\mL}$. These maps are not in general inverse of each other; for example it may very well be that different foliations give the same local system. 
 The following proposition is an example of this and will also be useful later.

\begin{prop}
	\label{uguaglianzalocalsys}
Take $\sF\subseteq T_{X/B}\subset T_X$ be an algebraically integrable foliation and $\pi$ the associated map as in Diagram (\ref{diagfol}). If the general fiber $F$ of $\pi$ is regular, that is $h^0(F,\Omega^1_F)=0$, then $\mL_\sF=\mD_X$.
\end{prop}
\begin{proof}By Proposition \ref{stesso}, we use the interpretation of the local system as $\mD_W$, that is we consider the exact sequence (\ref{primaoccorrenza})
\begin{equation}
\label{seqlemma}
	0\to \sL^\vee\to \Omega^1_{U}\to \tilde{\sF}^\vee\to M'\to 0.
\end{equation}
It will be enough to show that $\tilde{f}_*\sL^\vee\cong\tilde{f}_*\Omega^1_{U}$. One inclusion is trivial, to prove the opposite, take a section $s\in \Gamma(A,\tilde{f}_*\Omega^1_{U})$ on an open $A\subset B$. This is of course a section in $\Gamma(\tilde{f}^{-1}(A),\Omega^1_{U})$ and we show that it goes to zero in $\Gamma(\tilde{f}^{-1}(A),\tilde{\sF}^\vee)$. Note that  $\tilde{\sF}^\vee$ is the double dual of $\Omega_{U/W}^1$ since Sequence \ref{seqlemma} is obtained by dualizing twice
$$
0\to \pi^*\Omega^1_W\to \Omega^1_U\to \Omega^1_{U/W}\to 0.
$$
Hence $\tilde{\sF}^\vee$ is torsion free and, restricted on the general fiber, coincides with the sheaf of 1-forms on such a fiber. The section $s$ restricts to a global 1-form on the general fiber, hence vanishes by hypothesis. Since $\tilde{\sF}^\vee$ is torsion free we conclude that $s$ is zero in $\Gamma(\tilde{f}^{-1}(A),\tilde{\sF}^\vee)$ therefore $s\in \Gamma(A,\tilde{f}_*\sL^\vee)$ and this concludes the proof.
\end{proof}

In the setting of this Proposition, we immediately see that the local systems $\mL_\sF=\alpha(\sF)$ and $\mL_{0}=\alpha(0)$ are both equal to $\mD_X$ (see Example \ref{casibanali}) even if the foliation $\sF$ is different from the trivial foliation.

Nevertheless in the next sections we will give some description of the fixed points of this correspondence, that is foliations $\sF$ with $\beta(\alpha(\sF))=\sF$ and local systems $\mL$ with $\alpha(\beta(\mL))=\mL$.

\section{Towers of fibrations}
\label{sez4}
In this section we consider towers of semistable fibrations over a smooth curve $B$ and we study the relation with local systems $\mL$ and foliations $\sF$.

\subsection{2-Towers}
Let $X$ and $Y$ be two smooth algebraic varieties of dimension $n$ and $m$ respectively. Let $f\colon X\to B$  be a semistable fibrations and $h\colon X\to Y$, $g\colon Y\to B$ fibrations such that $f=h\circ g$. That is we are considering the following situation 
\begin{equation}
\xymatrix{
X\ar[d]^{h}\ar@/_2pc/[dd]_{f}\\
Y\ar[d]^{g}\\
B
}
\end{equation}  that we call a $2$-Tower.

By Section \ref{sez2}, we have the local systems $\mD_X$ and $\mD_Y$ which we recall are defined by the exact sequences 
\begin{equation}
\label{dx1}
0\to \omega_B\to f_*\Omega^1_{X,d}\to \mD_X\to 0.
\end{equation}
 and 
\begin{equation}
\label{dy}
0\to \omega_B\to g_*\Omega^1_{Y,d}\to \mD_Y\to 0 
\end{equation} respectively.

The first easy relation between $\mD_X$ and $\mD_Y$ is given by the following proposition
\begin{prop}
In our setting, we have an inclusion of local systems $\mD_Y\leq\mD_X$.
\end{prop}
\begin{proof}
By the inclusion of sheaves $g_*\Omega^1_{Y,d}\subseteq f_*\Omega^1_{X,d}$ given by pullback via $h$, we immediately have the thesis comparing the above sequences (\ref{dx1}) and (\ref{dy})
\begin{equation}
\xymatrix{
0\ar[r]&\omega_B\ar[r]\ar@{=}[d]&f_*\Omega^1_{X,d}\ar[r]&\mD_X\ar[r]&0\\
0\ar[r]&\omega_B\ar[r]&g_*\Omega^1_{Y,d}\ar[r]\ar@{^{(}->}[u]&\mD_Y\ar[r]\ar@{^{(}->}[u]&0
}
\end{equation}
\end{proof}
Both Sequence (\ref{dx1}) and (\ref{dy}) split by \cite[Lemma 2.2]{RZ4} and the splitting on (\ref{dx1}) can be chosen in agreement with the one on (\ref{dy}).

A natural question is determining to what extent the local system $\mD_Y$ recovers the variety $Y$. More precisely, if we have a semistable fibration $f\colon X\to B$ and a local system $\mL\leq\mD_X$, under which hypothesis on $\mL$ and in what sense can we recover the variety $Y$ and the morphisms $g$ and $h$? 

To approach this problem we need the famous classical theory of Castelnuovo and de Franchis. We briefly recall that this result states that if $S$ is a smooth complex surface with two linearly independent one forms $\eta_1,\eta_2$ such that $\eta_1\wedge\eta_2=0$, then there exists a morphism onto a smooth curve $p\colon S\to C$ and furthermore $\eta_1,\eta_2$ are pullback via $p$ of one forms on $C$. 

This result has been later generalized for higher dimensional varieties by Catanese \cite[Theorem 1.14]{Ca2} and Ran \cite[Prop II.1]{Ran}. We also have proved a generalization and a relative version of this result in \cite[Theorem 5.6 and Theorem 5.8]{RZ4}.

 The idea behind all these generalizations is the following.
Let $X$ be an $n$-dimensional smooth variety and $\omega_1,\dots, \omega_l \in H^0 (X,\Omega^1_X)$ linearly independent 1-forms such that $\omega_{j_1}\wedge\dots\wedge \omega_{j_{k+1}}= 0$ for every $j_1,\dots,j_{k+1}$ and that no collection of $k$ linearly independent forms in the span of $\omega_1,\dots, \omega_{j_{k+1}}$ wedges to zero. Then there exists a holomorphic map $p\colon X\to Y$ over a $k$-dimensional normal variety $Y$ such that $\omega_i\in p^*H^0(Y,\Omega^1_Y)$. The crucial point in the proof is that these global 1-forms naturally define a foliation on $X$ whose leaves are closed on a good open set $X\setminus \Sigma$, with $\codim \Sigma\geq 2$. Furthermore on the universal covering $X_U\to X$ these forms are exact, $\omega_i=dF_i$, and the functions $F_i$ define a holomorphic map $\phi \colon X_U\to \mC^l$ constant on the leaves of the foliation. The action of the fundamental group of $X$ on the image of $\phi$ is induced by the effect of the deck transformations on the $F_i$ by 
\begin{equation}
	\label{action}
	F_i(\gamma x)=F_i(x)+c_\gamma
\end{equation}
 Hence the action of $\phi(X_U)$ is properly discontinuous and we get $Y$ as the normalization of the quotient.

Now note that given our local system $\mD_{X}$, the wedge product naturally gives a map $\bigwedge^{i} \mD_X\to  \mD_X^i$ because the wedge of closed forms is a closed form. Considering $\mL\leq\mD_X$ as above we naturally have the restrictions $\bigwedge^{i} \mL\to  \mD_X^i$ and inspired by the Castelnuovo-de Franchis result we give the following definition
\begin{defn}
	\label{deflocsys}
	We say that  $\mL\leq\mD_X$ is an {\it{order $m$ Castelnuovo generated}} local system if it is generated under the monodromy action by a vector space $V\leq\Gamma(A,\mL)$ on a contractible open set $A\subset B$ such that  $\dim V \geq m+1$, the map $\bigwedge^{m} V\to  \mD_X^m$ is zero while  $\bigwedge^{m-1} V\to  \mD_X^{m-1}$
	is injective on decomposable elements.
	
	We say that  $\mL$ is {\it{of Castelnuovo-type of order $m$}} if  $\rank\mL \geq m+1$, $\bigwedge^{m} \mL\to  \mD_X^m$ is zero while  $\bigwedge^{m-1} \mL\to  \mD_X^{m-1}$
	is injective on decomposable elements.
\end{defn}

We will need the following Lemma which has its own interest.
\begin{lem}
\label{tecnico}
If $s_1,\dots,s_{m+1}$ are 1-forms on $X$ such that the wedge product of every $m$-uple is an $m$-form which vanishes when restricted to the fibers of $f$, then the $m+1$-wedge product $s_1\wedge\dots\wedge s_{m+1}$ is  zero on $X$.
\end{lem}
\begin{proof}
It is enough to prove this vanishing on a suitable open subset of $X$. Consider local coordinates $x_1,\dots,x_{n-1},t$, with $t$ being the variable on the base $B$. The wedge product $s_1\wedge\dots\wedge s_{m+1}$ is an $m+1$-form locally given by the $m+1\times m+1$ minors of the $m+1\times n$ matrix obtained by the local expression of the $s_i$.
The hypothesis that all the possible $m$-wedge products are zero when restricted to the fibers means that all the $m\times m$ minors where $dt$ does not appear are zero. From this it easily follows that all the $m+1\times m+1$ minors are zero. 
\end{proof}

The first result is the following
\begin{thm}
\label{castel}
Let $f\colon X\to B$ be a semistable fibration and $\mL\leq\mD_X$ an order $m$ Castelnuovo generated local system. Then up to a base change $\tilde{f}\colon \widetilde{X}\to \widetilde{B}$ there exist a normal $m$-dimensional complex space $\widetilde{Y}$ and holomorphic maps $\tilde{g}\colon\widetilde{Y}\to\widetilde{B}  $ and $\tilde{h}\colon \widetilde{X}\to\widetilde{B}$ such that $\tilde{f}=\tilde{h}\circ \tilde{g}$.
\end{thm}

\begin{proof}

To the local system $\mL$ we associate its monodromy group as follows. The action $\rho$ of the fundamental group $\pi_1(B, b)$ on the stalk of $\mD_X$ restricts to an action $\rho_\mL$ on the stalk of $\mL$. Denote $\ker \rho_\mL$ by $H_\mL$ and consider $\widetilde{B}\to B$ the covering classified by $H_\mL$. Of course $\tilde{f}\colon \widetilde{X}\to \widetilde{B}$ denotes the associated base change.

The inverse image of the local system $\mL$ on $\widetilde{B}$ is trivial and so the lifting of the sections of $\mL$, obtained as recalled above by \cite[Lemma 2.2]{RZ4}, are global 1-forms in $\tilde{f}_*\Omega^{1}_{\widetilde{X},d}$, in particular $V\leq H^0(\widetilde{B},\tilde{f}_*\Omega^{1}_{\widetilde{X},d})$.

 The condition of triviality of $\bigwedge^{m} V\to  \mD_X^m$ holds also on $\widetilde{X}$. Hence, as we have seen by Lemma \ref{tecnico}, we have the vanishing of $\bigwedge^{m+1} V\to\tilde{f}_*\Omega^{m+1}_{\widetilde{X},d}$ and we get a set of closed $1$-forms on $\widetilde{X}$ such that every possible $m+1$-wedge is zero. On the other hand, by the hypothesis on $\bigwedge^{m-1} V\to  \mD_X^{m-1}$, no collection of $m-1$ linearly independent forms wedges to zero.
Hence we have two cases, either no collection of $m$ linearly independent forms wedges to zero or there are at least $m$ of these forms with zero wedge.

In the first case by the Castelnuovo-de Franchis Theorem \cite[Theorem 1.14]{Ca2} there exist a foliation on $\widetilde{X}$ defined by the elements of $V$ which gives, as recalled above, a normal $m$-dimensional complex space $\widetilde{Y}$ and a morphism $\tilde{h}\colon  \widetilde{X}\to \widetilde{Y}$ such that all these global 1-forms on $\widetilde{X}$ are pullback via $\tilde{h}$ of 1-forms on $\widetilde{Y}$. 

Note that since the covering may not be finite, $\widetilde{B}$ and $\widetilde{X}$ may not be compact. Hence it is necessary that our 1-forms are closed so that they still define an integrable foliation.

To conclude our proof in this case we need to show that  $\tilde{f}$ factors via $\tilde{h}$. It is enough to consider the   morphism
	$$
	\widetilde{X}\times\widetilde{B}\xrightarrow{\tilde{h}\times\text{id}}\widetilde{Y}\times\widetilde{B}\xrightarrow{p}\widetilde{B}
	$$ where $p$ is the projection. Then $	\widetilde{X}$ is isomorphic to the incidence variety $I\subset 	\widetilde{X}\times\widetilde{B}$. Furthermore note that the hypotheses on $\bigwedge^{m} V\to  \mD_X^m$ and $\bigwedge^{m-1} V\to \mD_X^{m-1}$ ensure that if $\widetilde{F}$ is a fiber of $\tilde{f}$ then $\tilde{h}(\widetilde{F})$ is exactly a $(m-1)$-dimensional subvariety of $\widetilde{Y}$. In particular for $y\in \tilde{h}(\widetilde{F})$ we have that $\tilde{h}_{|\widetilde{F}}^{-1}(y)\subseteq \tilde{h}^{-1}(y)$ are of the same dimension. Since the fibers of $\tilde{h}$ are connected, we  have that $\tilde{h}_{|\widetilde{F}}^{-1}(y)= \tilde{h}^{-1}(y)$ and $\widetilde{Y}$ is isomorphic to the image $J:=(\tilde{h}\times\text{id})(I)$; we define $\tilde{g}$ as the restriction of $p$ to $J$.

In the second case, that is when there are at least $m$ forms among the liftings of $V$ with zero wedge, we can also apply the Castelnuovo-de Franchis Theorem \cite[Theorem 1.14]{Ca2}. In this case however we obtain a map $\tilde{h}'$ to a $m-1$-dimensional variety $\widetilde{Z}$. By our hypothesis on $\bigwedge^{m-1} V\to \mD_X^{m-1}$ it immediately follows that the restriction $\tilde{h}'_{|\widetilde{F}}$ is surjective onto $\widetilde{Z}$. Hence we define $\widetilde{Y}$ as the Stein factorization of $\tilde{h}'\times\text{id}$ onto $\widetilde{Z}\times \widetilde{B}$. Note that $\tilde{g}\colon \widetilde{Y}\to \widetilde{B}$ is immediately induced by the projection on $\widetilde{B}$.
\end{proof}

Clearly the condition on $\bigwedge^{m-1} V\to \mD_{X}^{m-1}$ fixes the dimension of $\widetilde{Y}$ to be exactly $m$. Removing this condition we obtain the following corollary:
\begin{cor}
	Let $f\colon X\to B$ be a semistable fibration and $\mL\leq\mD_X$ a local system generated by $V$ with $\dim V \geq m+1$ and such that the map $\bigwedge^{m} V\to  \mD_X^m$ is zero. Then up to a base change $\tilde{f}\colon \widetilde{X}\to \widetilde{B}$ there exist a normal complex space $\widetilde{Y}$ of dimension $\leq m$ and morphisms $\tilde{g}$ and $\tilde{h}$ such that $\tilde{f}=\tilde{h}\circ \tilde{g}$.
\end{cor}

\begin{rmk}
	\label{algebrico}
	As highlighted in the proof of Theorem \ref{castel}, the monodromy group $G_\mL=\pi_1(B,b)/H_\mL$ may not be finite; in this case $ \widetilde{X}, \widetilde{Y},\widetilde{B}$ are not algebraic varieties but only complex spaces. On the other hand note that the fibers $\widetilde{F}$ are algebraic and so are their images $\tilde{h}(\widetilde{F})$ (by \cite[Theorem 5.6]{RZ4} they are actually varieties of general type) and the holomorphic map $ \tilde{h}_{|\widetilde{F}}$ is a morphism in the algebraic category. In particular $T_{\widetilde{X}/\widetilde{Y}}$ is an algebraically integrable foliation.
	
	If the monodromy group $G_\mL=\pi_1(B,b)/H_\mL$ is finite then everything is in the algebraic category.
\end{rmk}

\begin{rmk} Note that if we call as before $\phi \colon X_U\to \mC^l$ the map constant on the leaves of the foliation induced by the elements of the vector space $V$ given in Definition \ref{deflocsys}, we have in principle that only $\pi_1(\widetilde{X})$ acts on the image of $\phi$. Furthermore this action is properly discontinuous.

 In fact since $V$ is not necessarily closed under the monodromy action, it follows that if $\sigma\in V$ it may happen that $g(\sigma)\notin V$ for some $g\in \pi_1(B)$. So even if it is true that $\sigma=dF$ and $g(\sigma)=dH$ on the universal covering $X_U$, $g(\sigma)$ is not necessarily zero on the foliation and $H$ is not necessarily constant on the leaves. Hence (\ref{action}) becomes 
\begin{equation}
	\label{noazione}
	F(\gamma x)=H(x)+c_\gamma
\end{equation} since $d (F\circ\gamma) =d (\gamma^* F)=\gamma^*d  F=g(\sigma)=dH$.
Equation (\ref{noazione}) does not define an action on  the image of $\phi$ for an element $\gamma\in \pi_1({X})$ corresponding to $g$ via $\pi_1(X)\to \pi_1(B)$. On the other hand if $\gamma'\in \pi_1(\widetilde{X})< \pi_1({X})$ corresponds to $g'\in \pi_1(\widetilde{B})$ then the action of $\gamma'$ is properly discontinuous as in the usual case. See Figure \ref{fig1}.
\end{rmk}
\begin{figure}[h!]
\includegraphics[scale=.9]{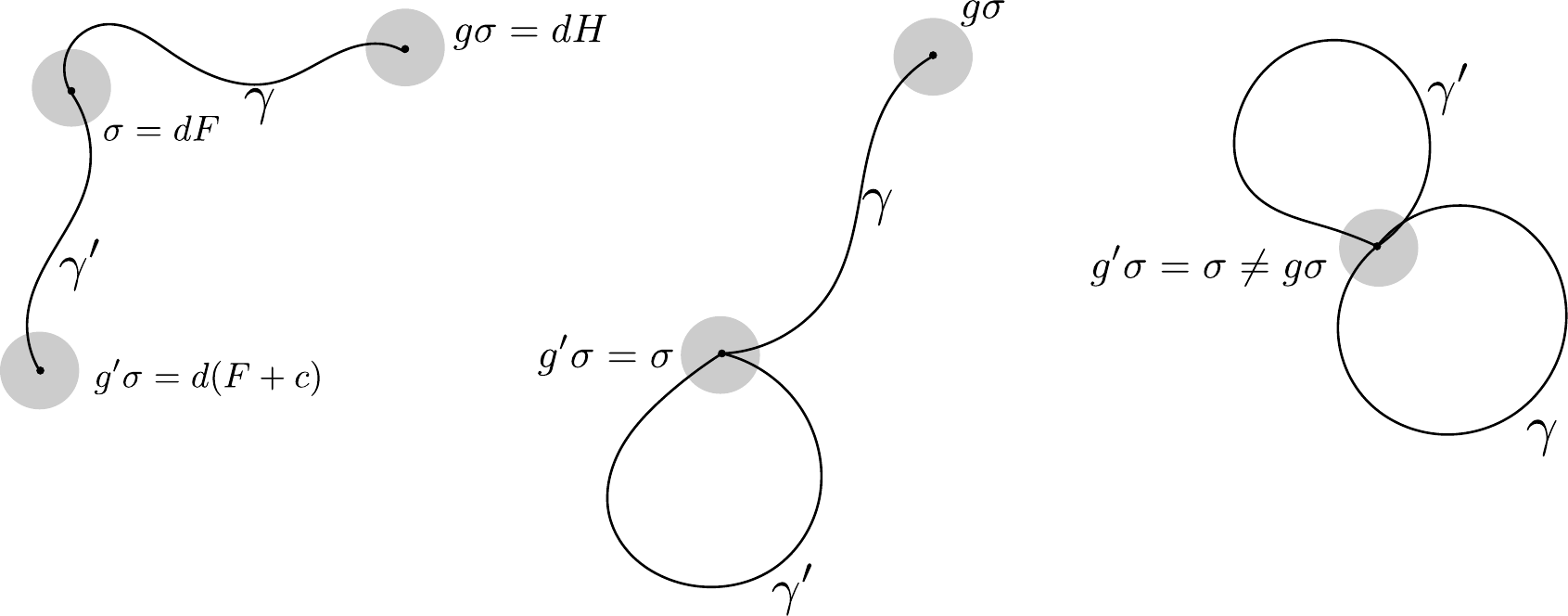}
\caption{The paths $\gamma$ and $\gamma'$ on, from left to right, $X_U$, $\widetilde{X}$, and $X$.}	\label{fig1}
\end{figure}

\begin{rmk} In the case where $\mL$ is of Castelnuovo type on the other hand this does not happen and we have an action of $\pi_1(X)$ on $\phi(X_U)$, but not necessarily properly discontinuous. Indeed, if $\sigma\in \Gamma(A,\mL)$ is a local section,  $g(\sigma)$ vanishes on the foliation and $H$ is constant on the leaves. So if the element $\gamma'$
 is in the fundamental group  $\pi_1(\widetilde{X})< \pi_1({X})$ then the actions is the same as in (\ref{action}) and it is properly discontinuous. On the other hand, for $\gamma$ not in $\pi_1(\widetilde{X})$, we can only say that it is given by an affinity of $\mC^l$, not necessarily a translation, and it may be not properly discontinuous (this affinity is determined by the matrix of the monodromy action of $g$ on the stalk of the local system $\mL$). 
 \end{rmk}
 
In the case where $\mL$ is of Castelnuovo type we can refine Theorem \ref{castel} as follows.
\begin{cor}
	\label{castel1}
	Let $f\colon X\to B$ be a semistable fibration and $\mL\leq \mD_X$ a local system of Castelnuovo-type of order $m$. Then up to a base change $\tilde{f}\colon \widetilde{X}\to \widetilde{B}$ there exist a normal $m$-dimensional complex space $\widetilde{Y}$ and holomorphic maps $\tilde{g}$ and $\tilde{h}$ as in Theorem \ref{castel} such that $\tilde{f}=\tilde{h}\circ \tilde{g}$. Furthermore if  $\sF_{\mL}$ is the foliation associated to $\mL$ then it is algebraically integrable.
\end{cor}
\begin{proof}
	The existence of $\widetilde{Y}$ follows by Theorem \ref{castel} because of course Castelnuovo type implies Castelnuovo generated. Hence we concentrate on the algebraic integrability of  $\sF_{\mL}$.
	
If we denote by  $p\colon\widetilde{X}\to X$  the covering then we have the inclusion $T_{\widetilde{X}/\widetilde{Y}}\subseteq p^*\sF_{\mL}$ on an open subset. Furthermore in this setting they are foliations of the same rank $n-m$. Hence it is not difficult to see that  $T_{\widetilde{X}/\widetilde{Y}}=p^*\sF_{\mL}$.
By Remark \ref{algebrico}, $T_{\widetilde{X}/\widetilde{Y}}$ is algebraically integrable, hence $\sF_{\mL}$  is also algebraically integrable. 

Note that this argument is not true in the setting of Theorem \ref{castel} where $T_{\widetilde{X}/\widetilde{Y}}$ is the foliation given by $V$ and hence $p^*\sF_{\mL}$ can be in principle of rank strictly smaller than $T_{\widetilde{X}/\widetilde{Y}}$. 
\end{proof}

By these results, a local system $\mL$ of Castelnuovo-type makes it possible to reconstruct an intermediate variety up to a covering of the base $B$. On the other hand by Lemma \ref{lasic} we know that also algebraically integrable foliations give a factorization of $f$ up to birational morphism, as in Diagram \ref{diagfol}.
Hence we now correlate this two approaches.

\begin{thm}
	\label{dxsx}
Let $f\colon X\to B$ be a semistable fibration and $\mL$ a Castelnuovo-type local system. Then $\mL$ gives normal variety $Y'$ and complex spaces $\widetilde{Y}$, and $\widehat{Y}$, all of the same dimension, which fit into the following diagram
\begin{equation}
	\label{diagdxsx}
	\xymatrix{
		\widetilde{X}\ar[d]\ar[r]& X\ar[dd]& X'\ar[d]\ar[l]& \widehat{X}\ar[l]\ar[d]\ar@/_1.5pc/[lll]\\
		\widetilde{Y}\ar[d]& & Y'\ar[d]& \widehat{Y}\ar[d]\ar@{-->}[l]\ar@{-->}@/_1.5pc/[lll]\\
		\widetilde{B}\ar[r]& B& B\ar@{=}[l]& \widetilde{B}\ar[l]\ar@{=}@/^1.3pc/[lll]
	}
\end{equation}
\newline
where $\widetilde{B}\to B$ is a covering, $\widehat{X}\to \widetilde{X}$ is a (generically) degree 1 map, $\widehat{Y}\dashrightarrow Y'$ is a dominant meromorphic map and $\widehat{Y}\dashrightarrow \widetilde{Y}$ is a bimeromorphic map.
\end{thm}

\begin{proof}
Denote $\sF:=\sF_\mL$ the foliation associated to $\mL$.
Applying Corollary \ref{castel1} we know that $\sF$ is algebraically integrable and that we have a covering $\widetilde{B}\to B$ of $B$ and a normal complex space $\widetilde{Y}$ which factors the base change of $\widetilde{X}\to \widetilde{B}$.

We can also apply Lemma \ref{lasic} to the foliation $\sF$ and this will give a birational morphism $X'\to X$ and a variety $Y'$ which factors the fibration $X'\to B$. 

That is, putting these two together we have the diagram 
\begin{equation}
	\label{diagdd}
\xymatrix{
\widetilde{X}\ar[d]\ar[r]& X\ar[dd]& X'\ar[d]\ar[l]\\
\widetilde{Y}\ar[d]& & Y'\ar[d]\\
\widetilde{B}\ar[r]& B& B\ar@{=}[l]
}
\end{equation}

Now apply again Theorem \ref{castel} to the same local system $\mL$ but this time looking at the right side of Diagram (\ref{diagdd}). We find again our covering $\widetilde{B}\to B$ and the pullback $\widehat{X}\to \widetilde{B}$ factors via a normal complex space $\widehat{Y}$. 
\begin{equation}
\xymatrix{
\widetilde{X}\ar[d]\ar[r]& X\ar[dd]& X'\ar[d]\ar[l]& \widehat{X}\ar[d]\ar[l]\\
\widetilde{Y}\ar[d]& & Y'\ar[d]& \widehat{Y}\ar[d]\\
\widetilde{B}\ar[r]& B& B\ar@{=}[l]& \widetilde{B}\ar[l]
}
\end{equation}

Noticing that $\widetilde{X}=X\times_B\widetilde{B}$ and $\widehat{X}=X'\times_B\widetilde{B}$, from the birational morphism $X'\to X$ we immediately obtain a generically degree 1 morphism $\widehat{X}\to \widetilde{X}$. Finally since both $\widehat{Y}$ and $\widetilde{Y}$ are obtained by the application of  Theorem \ref{castel}, we also have a bimeromorphic map $\widehat{Y}\dashrightarrow \widetilde{Y}$. On the other hand it is clear that we have just a meromorphic map, not necessarily holomorphic, between $\widehat{Y}$ and $Y'$. The following diagram sums up the situation and gives the final result.
\begin{equation}
\xymatrix{
\widetilde{X}\ar[d]\ar[r]& X\ar[dd]& X'\ar[d]\ar[l]& \widehat{X}\ar[l]\ar[d]\ar@/_1.5pc/[lll]\\
\widetilde{Y}\ar[d]& & Y'\ar[d]& \widehat{Y}\ar[d]\ar@{-->}[l]\ar@{-->}@/_1.5pc/[lll]\\
\widetilde{B}\ar[r]& B& B\ar@{=}[l]& \widetilde{B}\ar[l]\ar@{=}@/^1.3pc/[lll]
}
\end{equation}
\end{proof}

Using the same notation as before, note that the map  $\widehat{Y}\dashrightarrow Y'$ is basically given by the quotient of the non properly discontinuous action of $\pi_1(X)$ on $\phi(X_U)$.

Motivated by this result we give the following definition
\begin{defn}\label{casttypef}
	We say that an algebraically integrable foliation $\sF$ of rank $n-m$ on $X$ is of Castelnuovo-type if the associated local system $\mL_\sF$ is of Castelnuovo-type of order $m$.
\end{defn}

\subsection{$s$-Towers}

We have seen that a 2-Tower $X\to Y\to B$ gives an inclusion of local systems on $B$, $\mD_Y\leq\mD_X$. Viceversa Theorem \ref{castel} and Corollary \ref{castel1} say that an inclusion of local systems $\mL\leq\mD_X$ gives, under suitable hypothesis on $\mL$, an intermediate variety $\widetilde{Y}$ which factors the pullback of the fibration $\widetilde{X}\to \widetilde{B}$.

In the same way it is clear that given an $s$-Tower 
\begin{equation}
X\to Y_1\to Y_2\to \dots \to Y_s\to B
\end{equation} we have a flag of local subsystems on $B$
\begin{equation}
\mD_{Y_s}\leq\dots\leq\mD_{Y_2}\leq\mD_{Y_1}\leq\mD_{X}
\end{equation}
The following result gives the viceversa
\begin{cor}
	\label{flag}
Let $X\to B$ be a semistable fibration and $\mL_{s}<\dots<\mL_{2}<\mL_{1}<\mD_{X}$ a flag of local systems on $B$ such that 
\begin{equation}
\bigwedge^{m_i}\mL_i\to \mD_X^{m_i}
\end{equation} is zero for every $i$ while the quotient $\mL_i/\mL_{i+1}$ has rank at least $m_i$ and every $m_i-1$-uple of sections in $\mL_i\setminus\mL_{i+1}$ does not wedge to zero in $\mD_X^{m_i-1}$. Then there exists a covering $\widetilde{B}\to B$ and a factorization of the pullback $\widetilde{X}\to\widetilde{B}$ as
\begin{equation}
\xymatrix{
&\widetilde{X}\ar[ld]\ar[dd]\ar[ddddrr]&&\\
\widetilde{Y_1}\ar@{-->}[rd]\ar[ddddrr]&&&\\
&\widetilde{Y_2}\ar@{-->}[rd]\ar[dddr]&&\\
&&\ddots\ar@{-->}[rd]&\\
&&&\widetilde{Y_s}\ar[dl]\\
&&\widetilde{B}&
}
\end{equation}
where the $\widetilde{Y_i}$ are normal complex spaces with meromorphic maps between them.
\end{cor}
\begin{proof}
We consider the covering $\widetilde{B}\to B$ which trivializes $\mL_1$, hence all the $\mL_i$. By the hypothesis on the local systems $\mL_i$ we can easily find using Theorem \ref{castel} spaces $\widetilde{Y_i}$ such that 
$\widetilde{X}\to\widetilde{B}$ factors as 
$\widetilde{X}\to\widetilde{Y_i}\to\widetilde{B}$.

As the last step we note that on the smooth part of every $\widetilde{Y_i}$ we can define a morphism to $\widetilde{Y}_{i+1}$. This comes from the fact that $\mL_{i+i}<\mL_i$, hence over smooth points of $\widetilde{Y_i}$ we have that $T_{\widetilde{X}/\widetilde{Y}_{i}}\subset T_{\widetilde{X}/\widetilde{Y}_{i+1}}$, hence in this open set every fiber of $\widetilde{X}\to \widetilde{Y_i}$  is contained in a fiber of $\widetilde{X}\to\widetilde{Y}_{i+1}$.
\end{proof}

\section{Foliations-Local systems: on the fixed points of the correspondence.}
This section is dedicated to the study of the fixed points of the correspondence highlighted in (\ref{corrisponde}). We show that foliations and local systems of Castelnuovo type are indeed fixed points.

\begin{thm}
	\label{fissi}
	Let $\mL\leq \mD_X$ and $\sF\subseteq T_{X/B}$ be respectively a local system and an algebraically integrable foliation of Castelnuovo type of order $m$.
	Furthermore assume that $\mL$ is maximal with the property that  $\bigwedge^{m} \mL\to  \mD_X^m$ is zero. Then they are fixed points of the correspondence, that is $\alpha(\beta(\mL))=\mL$ and $\beta(\alpha(\sF))=\sF$. 
\end{thm}
\begin{proof}
	We divide the proof in two parts, first for the local system and then for the foliation.
	
\textit{The local system.} 
Take $\mL$ as in the statement. First note that the foliation $\beta(\mL)$ is algebraically integrable by Corollary \ref{castel1}.

 Hence  $\beta(\mL)$ gives a diagram 
\begin{equation}
	\xymatrix{
		U\ar[d]_{e}\ar[r]^{\pi}&W\ar[ddl]\\
		X\ar[d]_{f}&\\
		B
	}
\end{equation} with $\dim W=m$, as in (\ref{diagfol}) and $\alpha(\beta(\mL))=\mD_W$ using the notation of Section \ref{sez3}. Hence by the maximality hypothesis it will be enough to show that $\mL\leq\mD_W$.
Now recall that $\mL$ and $\mD_W$ by definition fit into the exact sequences 
\begin{equation}
	\xymatrix{
		0\ar[r]&\omega_B\ar[r]&\tilde{f}_*\sL^\vee_{d}\ar[r]&\mD_{W}\ar[r]&0\\
		0\ar[r]&\omega_B\ar[r]\ar@{=}[u]&\sS\ar[r]&\mL\ar[r]&0
	}
\end{equation} $\tilde{f}=f\circ e$, see Diagrams (\ref{ll}) and (\ref{ss}). So our problem is reduced to showing that $\sS\subseteq \tilde{f}_*\sL^\vee_{d}\subseteq f_*\Omega_{X,d}^1$. It is not difficult to see that this is true since the foliation $\beta(\mL)$ is defined via (\ref{defF}), coincides with $T_{U/W}$ on an open set $X^0\subset X$ such that ${\rm{codim}}(X\setminus X_0)\geq 2$, and $\sL$ is 
\begin{equation}
	0\to T_{U/W}\to T_U\to \sL\to 0
\end{equation} see the Sequence (\ref{due}). 

\textit{The foliation.} Take $\sF$ as in the statement. Once again we know that by definition $\alpha(\sF)=\mD_{W}$ hence $\beta(\alpha(\sF))$ is obtained as in (\ref{defF}) with $\sS=\tilde{f}_*\sL^\vee_{d}$. It immediately follows that $\sF\subseteq \beta(\alpha(\sF))$ on a good open subset hence they coincide since they are foliations of the same rank $n-m$.
\end{proof}

\begin{rmk}
	Note that by the above Theorem \ref{fissi}, Theorem \ref{dxsx} can also be stated in terms of foliations; that is if $\sF$ is a foliation of Castelnuovo-type then it uniquely corresponds to the local system $\mL_{\sF}$ and they give the varieties $Y',\widetilde{Y},\widehat{Y}$ as in Diagram (\ref{diagdxsx}).
\end{rmk}

\begin{rmk}Note that viceversa it is not true that a local system $\mL$ with $\alpha(\beta(\mL))=\mL$, and minimal with this property, is of Castelnuovo type, or even Castelnuovo generated. In fact even assuming that $\sF_{\mL}$ is algebraically integrable, it follows that $\alpha(\beta(\mL))=\alpha(\sF_{\mL})=\mD_W$ with the same notation of the previous proof. Hence if we assume that $\mL$ is a fixed point then we certainly have $\mL=\mD_{W}$. Now denote by  $m$ the dimension of the variety $W$. Hence $\bigwedge^{m} \mL\to  \mD_X^m$ is zero by dimensional reasons but we can not be sure that $\bigwedge^{m-1}\mL\to  \mD_X^{m-1}$ is injective on decomposable elements. Of course there exists $V'<\Gamma(A,\mL)$ and an integer $m'<m$ such that	$\bigwedge^{m'} V'\to  \mD_X^m$ is zero and $\bigwedge^{m'-1} V'\to  \mD_X^{m-1}$ is injective on decomposable elements. But the local system generated by $V'$ might not be a fixed point because $\dim V'$  is not necessarily $\geq m'+1$ hence it is not of Castelnuovo type (not even Castelnuovo generated) and Theorem \ref{fissi} does not apply (the associated foliation  may be not algebraically integrable).
\end{rmk}

\begin{prob}
For the viceversa, find conditions that imply the algebraicity of $\sF_{\mL}$ at least for those $\mL$ such that $\alpha(\beta(\mL))=\mL$. We plan to study this problem in the next future.
\end{prob}

\section{MRC fibrations and Castelnuovo-type local systems}
\label{sez5}
In this section we use local systems and foliations side by side to study some problems on Harder-Narasimhan filtrations and rational connectedness.
We recall the key definitions and ideas.
Let $X$ be a smooth projective variety and $\alpha$ an ample class. Given a coherent torsion-free sheaf $\sE$ on X, the slope $\mu(\sE)$ with respect to the class $\alpha$ is defined as the ratio
	$$
	\mu(\sE)=\frac{\deg \sE}{\rank \sE}
	$$ where $\deg\sE=\sE\cdot \alpha^{n-1}.$

A torsion-free sheaf $\sE$ is said to be $\mu$–stable (respectively, $\mu$–semistable) if $\mu(\sF) <
\mu(\sE)$ (respectively, $\mu(\sF) \leq
\mu(\sE)$) for every proper coherent torsion-free subsheaf $\sF \subset \sE$.

We recall the well known Harder-Narasimhan filtration of $\sE$
\begin{thm}
Let X be a smooth projective variety and let $\sE$
be a torsion-free coherent sheaf of positive rank on $X$. Then there exists a unique
Harder-Narasimhan filtration, that is, a filtration $0= \sE_0 \subsetneq \sE_1 \subsetneq\dots\subsetneq \sE_r =\sE$ where
each quotient $\sQ_i := \sE_i/\sE_{i-1}$ is torsion-free, $\mu$-semistable, and where the sequence of
slopes $\mu(\sQ_i)$ is strictly decreasing.
\end{thm}

We recall that a variety is rationally connected if every two general points can be connected by a rational curve. It is well known that if we consider the Harder-Narasimhan filtration of the tangent sheaf $T_X$,
$$
0= \sF_0 \subsetneq \sF_1 \subsetneq\dots\subsetneq \sF_r =T_X
$$ and we set
$$
i_{\max}=\max\{0<i<k\mid \mu(\sQ_i)>0\}\cup \{0\} 
$$
if we assume that $i_{\max} >0$ then for every index 
$0< i\leq i_{\max} $, $\sF_i$ is an algebraically integrable foliation with rationally connected general leaf. In particular we have the following diagram of rational maps 
\begin{equation}
	\label{hnmrc}
\xymatrix{
&X\ar@{-->}^{\pi_1}[ld]\ar@{-->}^{\pi_2}[dd]\ar@{-->}^{\pi_{i_{\max }}}[ddddrr]&&\\
W_1\ar@{-->}[rd]&&&\\
&W_2\ar@{-->}[rd]&&\\
&&\ddots\ar@{-->}[rd]&\\
&&&W_{i_{\max}}
}
\end{equation} 
This means that the closure of the general fiber of the map $\pi_i$ is a rationally connected variety and it makes sense to study when $\pi_{i_{\max}}$ is a \emph{maximal} rationally connected fibration. This problem is posed, for example, in \cite{K}. 

The definition of maximal rationally connected fibration (usually abbreviated by MRC) is the following.
\begin{defn}\label{definitionMRC}
A dominant
rational map $X\dashrightarrow Z$ is MRC if there exist an open subset $X_0$ of $X$ and an open subset $Z_0$ of $Z$ such
that
\begin{itemize}
\item[a)] the induced map $X_0\to Z_0$ is
a proper morphism,
\item[b)] a general fiber is irreducible and rationally connected,
\item[c)] all rational curves which meet a general fiber $F$ are contained in $F$
\end{itemize}
\end{defn}
It is known that given a variety $X$, then a MRC fibration  $X\dashrightarrow Z$ always exists and it is unique up to birational morphism, see for example \cite[Theorem 4.8]{L}.

We also have a notion of relative MRC fibration as follows.
\begin{defn}
	Let $f \colon X \to Y$ be a morphism between smooth projective varieties. A relative MRC fibration is a dominant rational map $ \pi\colon X\dashrightarrow Z$ to a smooth projective variety $Z$ over $ Y$ such that for a general point $y\in Y$, the induced map $\pi_y\colon X_y \dashrightarrow Z_y$ is an MRC fibration.
\end{defn}
This also always exists, see \cite[Theorem 4.20]{L}.

In our relative setting $X\to B$ note that, since $B$ is a curve, the two notions coincide if the genus of $B$ is $g(B)>0$.

We consider the Harder-Narasimhan filtration of the relative tangent sheaf $T_{X/B}$
$$
0= \sF_0 \subsetneq \sF_1 \subsetneq\dots\subsetneq \sF_r =T_{X/B}
$$ 
The $W_i$ constructed as above in Diagram (\ref{hnmrc}), now fit a diagram as follows
\begin{equation}
	\label{hnmrcB}
	\xymatrix{
		X\ar@{-->}_{\pi_1}[rd]\ar[dddd]\ar@{-->}^{\pi_{i_{\max}}}@/^1.5pc/[dddrrr]&&&\\
		&W_1\ar@{-->}[rd]\ar[dddl]&&\\
		&&\ddots\ar@{-->}[rd]&\\
		&&&W_{i_{\max}}\ar[dlll]\\
		B
	}
\end{equation} 
While in general it is not known when $\pi_{i_{\max}}$ is MRC fibration, in our case, helped by the study of the local systems, we can prove the following
\begin{thm}
	\label{mrc}
	Let $f\colon X\to B$ be a semistable fibration and  
	$$
	0= \sF_0 \subsetneq \sF_1 \subsetneq\dots\subsetneq \sF_r =T_{X/B}
	$$ the Harder-Narasimhan filtration of the relative tangent sheaf of $f$ determined by a suitable polarization.
	Assume that $\mD_{X}$ is of Castelnuovo-type of order $m$.
	If the foliation $\sF_{i_{\max}}$ has generic rank $n-m$, then $X\dashrightarrow W_{i_{\max}}$ is the relative MRC fibration of $f$. Furthermore if $B$ is not $\mP^1$, $X\dashrightarrow W_{i_{\max}}$ is the MRC fibration of $X$.
\end{thm}

\begin{proof}
We know that the relative MRC fibration $X\dashrightarrow W$ exists. By the universal property we have a map $W_{i_{\max}}\dashrightarrow W$ and we want to see that it is birational. Call as before $\sF_{i_{\max}}$ and $\sF$ the foliations associated to $X\dashrightarrow W_{i_{\max}}$ and $X\dashrightarrow W$ respectively. Since the general fiber of these maps is rationally connected, in particular it is regular, hence by Proposition \ref{uguaglianzalocalsys}, we easily deduce that the associated local systems coincide, that is 
$$
\mD_{X}=\mL_\sF=\mL_{\sF_{i_{\max}}}.
$$

By the hypothesis on the rank of $\sF_{i_{\max}}$ we know that $\dim W_{i_{\max}}=m$.
On the other hand we have just proved  that $\mD_{X}=\mL_\sF$, hence we can apply Theorem \ref{dxsx} to this local system. Note that $W$ plays the role of $Y'$ in this theorem, hence we obtain a meromorphic map $\widehat{W}\dashrightarrow W$ with $\dim W=\dim\widehat{W}$ and a bimeromorphic map $\widehat{W}\dashrightarrow \widetilde{W}$ where $\widetilde{W}$ by Theorem \ref{castel} has dimension $m$.   This shows that $\dim W= \dim W_{i_{\max}}$ and this concludes the proof.

Alternatively note that $\sF_{i_{\max}}$ is of Castelnuovo type according to Definition \ref{casttypef}, hence by Theorem \ref{fissi} it is a  fixed points of the correspondence. We conclude that $\sF\subseteq  \beta(\alpha(\sF))=\beta(\mD_{X}) =\beta(\alpha(\sF_{i_{\max}}))=\sF_{i_{\max}}$. From the rational map $W_{i_{\max}}\dashrightarrow W$ we also have the opposite inclusion hence $\sF_{i_{\max}}=\sF$ and $W_{i_{\max}}$ is birational to $W$.
\end{proof}

\section{A Castelnuovo-de Franchis theorem for $p$-forms}
\label{sez6}

The generalizations of the Castelnuovo-de Franchis theorem mentioned in the previous sections only consider $1$-forms, in this section we study a more general case that will involve the study of $p$ forms.

\subsection{New notion of strictness and Castelnuovo-de Franchis theorem for $p$-forms}
We need a new setup.
Let $X$ be an $n$-dimensional smooth variety and $\omega_1,\dots, \omega_l \in H^0 (X,\Omega^p_X)$, $l\geq p+1$, linearly independent $p$-forms such that $\omega_i\wedge\omega_j=0$ (as an element of $\bigwedge^2\Omega^p_X$) for any choice of $i,j$. If this is the case then obviously the $\omega_i$ generate a subsheaf of $\Omega^p_X$ generically of rank 1; more concretely the quotients $\omega_i/\omega_j$ define global meromorphic functions on $X$. By taking the differential $d (\omega_i/\omega_j)$ we then get global meromorphic 1-forms on $X$ and the conditions we are interested in involves these 1-forms as follows.

\begin{defn}\label{definizionestrict}
We say that the set $\{\omega_1,\dots, \omega_l\}$ as above is  $p$-strict if there are $p$ of these meromorphic differential forms $d (\omega_i/\omega_j)$ that do not wedge to zero. We say that the vector subspace generated by $\{\omega_1,\dots, \omega_l\}$ is a $p$-strict vector subspace.
\end{defn}

The choice of the term strict in this definition comes from the fact that this condition is analogous to the strictness condition considered in \cite[Definition 2.1 and 2.2]{Ca2}, \cite[Definition 4.4]{RZ4}.

\begin{thm}
	\label{cast2}
	Let $X$ be an $n$-dimensional smooth variety and let $\{\omega_1,\dots, \omega_l \}\subset H^0 (X,\Omega^p_X)$ be a $p$-strict set. Then there exists a rational map $f\colon X\dashrightarrow Y$ over a $p$-dimensional smooth variety $Y$ such that the $\omega_i$ are pullback of some meromorphic $p$-forms $\eta_i$ on $Y$, $\omega_i=f^*\eta_i$, where $i=1,\cdots , l$.
\end{thm}
\begin{proof}
Each $p$-form $\omega_i$ determines by contraction a homomorphism
$$
T_X\to \Omega^{p-1}_X
$$ and we denote by $\sF'$ the intersections of these kernels. 
Since the forms $\omega_i$ are global holomorphic forms hence closed, $\sF'$ is closed under Lie bracket. 
As we showed in Subsection \ref{localtofoliation}, the saturation $\sF$ of $\sF'$ is a foliation and it turns out to be reflexive (saturated inside reflexive is reflexive), and locally free outside a set $S$ of codimension at least 2.

We consider then $\mC(\sF)$ the field of meromorphic functions on $X$ which are constant on the leaves of $\sF$. We take a smooth birational model $Y$ for $\mC(\sF)$. From $\mC(\sF)\subset \mC(X)$ we get a rational morphism 
$$
f\colon X\dashrightarrow Y.
$$ Since $\sF$ is generically of rank $n-p$, $\dim Y\leq p$; we will prove that the dimension of $Y$ is exactly $p$.

 We consider the good open set $U$ where $f$ is a holomorphic submersion and not all the $\omega_i$ vanish.
By our hypothesis of strictness of the $\omega_i$, there exists a point $x\in U$ and $p$ meromorphic functions $g_i$ defined as 
$$
g_i=\omega_{s_i}/\omega_{t_i}
$$ for $i=1,\dots,p$ and certain $\omega_{s_i}, \omega_{t_i}$ such that 
$$
dg_1\wedge dg_2\wedge\dots\wedge dg_p\neq 0 
$$ at $x$.

Now consider, without loss of generality, the condition $g_1=\omega_{s_1}/\omega_{t_1}$, which written as $\omega_{t_1}g_1=\omega_{s_1}$ gives
$$
\omega_{t_1}\wedge dg_1=0 
$$ by the fact that the $\omega_i$ are all closed. This, together with the hypothesis that $\omega_i\wedge\omega_j=0$ for every $i,j$, implies that  on a suitable open subset the meromorphic section of $\bigwedge^2\Omega^p_X$ given by
$$
\omega_{t_1}\bigwedge dg_1\wedge dg_2\wedge\dots\wedge dg_p=0
$$ is zero  (here of course we use that $dg_1\wedge dg_2\wedge\dots\wedge dg_p\neq0$). 

From $\omega_i\wedge\omega_j=0$ it follows that, again on an appropriate open subset of $X$, 
$$
\omega_i=f_idg_1\wedge dg_2\wedge\dots\wedge dg_p
$$ for $i=1,\dots,l$.
Since the forms $dg_1\wedge\dots\wedge \widehat{dg_j}\wedge\dots\wedge dg_p$ are independent, the 1-forms $dg_i$ vanish on the elements of $\sF$, hence the $g_i$ are  locally constant on the leaves of the foliation. By the fact that the $g_i$ are global meromorphic it follows that they are also constant on the closure of the leaves and hence they are elements $g_i\in \mC(\sF)$. The same is true for the $f_i$, since $d\omega_i=0$, hence $f_i\in \mC(\sF)$.

This exactly means that the forms $\omega_i$ are pullback of meromorphic forms on $Y$ and since the wedge $dg_1\wedge dg_2\wedge\dots\wedge dg_p$ is not zero, we have that $\dim Y=p$ and the leaves are closed on a good open subset of $X$.
\end{proof}
\begin{rmk}
We point out the main difference with the classical case of 1-forms, that is when $p=1$. In this case, if $\omega$ is a holomorphic 1-form on $X$, pullback on a suitable open subset of a meromorphic 1-form $\eta$ on $Y$, by a local computation it turns out that $\eta$ must also be holomorphic, as the pullback can not get rid of the poles of $\eta$. In the general case however, that is $p\geq2$, it may very well happen that the pullback of a meromorphic form on $Y$ is holomorphic on $X$, as the following local example easily shows.

Let consider $p=2$  and a map locally given in coordinates $x_1,x_2,x_3$ by
$$
f(x_1,x_2,x_3)=(x_1,x_1x_2).
$$ Denoting $y_1=x_1$ and $y_2=x_1x_2$ it immediately follows that the pullback of the meromorphic form $\frac{1}{y_1}dy_1\wedge dy_2$ is holomorphic:
$$
f^*\frac{1}{y_1}dy_1\wedge dy_2=dx_1\wedge dx_2.
$$ 
\end{rmk}

Of course in some cases it may still very well be that $f\colon X\dashrightarrow Y$ is actually a morphism.

Denote by $\sL\hookrightarrow \Omega_X^p$ the subsheaf  generated by the $p$-forms $\omega_i$. Under the hypotheses of the Castelnuovo theorem \ref{cast2} we have that $\sL$ has generic rank 1. Its double dual $\sL^{\vee\vee}$ is also of generic rank 1 but it is also reflexive, therefore it is an invertible sheaf by \cite[Proposition 1.9]{har}.
\subsection{Application to the Iitaka Fibration}

We show that under suitable hypothesis, the map given by the Castelnuovo Theorem \ref{cast2} is the Iitaka fibration. We start with a lemma which has its own interest.
\begin{lem}
\label{lemmasuL}
	If the p-forms $\omega_i$ satisfy the hypothesis of Theorem \ref{cast2} and furthermore they globally generate $\sL$, then $f\colon X\to Y$ is a holomorphic map onto a normal p-dimensional variety and furthermore the $\omega_i$ are pullback of global top forms on $Y$. 
\end{lem}
\begin{proof}
The foliation $\sF$ defined as in the proof of Theorem \ref{cast2}, that is as the kernel of these sections, is locally free. Furthermore the normal sheaf $N$ to the foliation given by 
\begin{equation}
	0\to \sF\to T_X\to N\to 0,
\end{equation} is also locally free. In this setting it is well known that the foliation is induced by a morphism $f\colon X\to Y$ onto a normal variety $Y$, for example see: \cite[Section 6]{Dr} and the therein quoted bibliography.

We have the following diagram
\begin{equation}
	\xymatrix{
		0\ar[r]&\sF\ar@{=}[d]\ar[r]&T_X\ar@{=}[d]\ar[r]&N\ar@{^{(}->}[d]\ar[r]&0\\
		0\ar[r]&T_{X/Y}\ar[r]&T_X\ar[r]&f^*T_Y\ar[r]&K\ar[r]&0
	}
\end{equation} In particular in the exact sequence
\begin{equation}
	0\to N\to f^*T_Y\to K\to 0 
\end{equation} $K$ is supported on a locus of codimension at least 2, by our hypothesis on the $\omega_i$'s. By dualisation it follows that:
$$
\det N^\vee\cong f^*\omega_Y^{\vee\vee}.
$$ Now it is not difficult to see that $\det N^\vee\cong \sL$. In fact the $\omega_i$'s, which generate $\sL$, are of course in $\det N^\vee$. Viceversa. Every section of $\det N^\vee$ vanishes on $\sF$, hence it is in the generated of the $\omega_i$'s, again by the fact that $\sF$ is defined as the kernel of these sections. Moreover for every point $x$ of $X$ there is at least one $\omega_i$ which is non zero at $x$. Hence we have 
$$
\sL\cong\det N^\vee\cong f^*\omega_Y^{\vee\vee}.
$$ We consider the spaces of global sections and by the projection formula we have:
$$
 H^0(X,\sL)=H^0(X,f^*\omega_Y^{\vee\vee})=H^0(Y,\omega_Y^{\vee\vee})=H^0(Y,\omega_Y)
$$ Hence the $\omega_i$ are pullback of top forms of $Y$.
\end{proof}

\begin{thm}\label{Iitakauno}
	Let $X$ be a smooth variety of dimension $n$. Assume that $\sL$ is globally generated by a $p$-strict subset. If the Kodaira dimension $\Kod X=\Kod \sL=p<n$ then the Stein factorization of $\varphi_{|\sL|}\colon X\to \mP(H^0(X,\sL))$ induces the $K_X$-Iitaka fibration.
\end{thm}

\begin{proof}
By the previous Corollary and by \cite[Theorem 2.1.33]{Laz} we have the diagram
	
\begin{equation}
\xymatrix{
&Y\ar[dl]&X\ar_{\phi_k}@{-->}[dd]\ar^{f}[l]&X_\infty\ar_{\phi_\infty}[dd]\ar^{u_\infty}[l]\ar_{f_\infty}@/_1.4pc/[ll]\\
\mP(H^0(X,\sL))&&&\\
&&Z_k&Z_\infty\ar_{\nu_k}[l]
}
\end{equation}
where $\phi_k$ is the map given by $|K_X|$ and $\phi_\infty$ is the $K_X$-Iitaka fibration.

By the previous Lemma \ref{lemmasuL}, we know that $\sL=f^*\omega_Y^{\vee\vee}$ hence $\Kod Y=\Kod\sL=p$ and $Y$ is of general type.

The very general fiber $X_z$ of $\phi_\infty$ is of Kodaira dimension 0 and we obtain a morphism
$$
f_\infty|_{X_z}\colon X_z\to Y
$$ between a variety of Kodaira dimension 0 to a variety of general type. There are two cases.
Assume $\dim f_\infty(X_z)>0$. Then $f_\infty(X_z)$ must be a variety of general type if $z$ is a general point in $Z_\infty$. This is a contradiction since $\Kod X_z=0$.

Hence $f_\infty(X_z)$ is zero dimensional for general $z$ is in $Z_\infty$. Since the general fiber of both $f_\infty$ and $\phi_\infty$ are irreducible we obtain a map $Z_\infty\to Y$, see the Rigidity Lemma in \cite[Lemma 1.6]{KM}. 

By the unicity of the Iitaka fibration up to birational modification we obtain 
\begin{equation}
	\xymatrix{
		&Y&X\ar_{\phi_k}@{-->}[d]\ar^{f}[l]&X_\infty\ar_{\phi_\infty}[d]\ar^{u_\infty}[l]\ar_{f_\infty}@/_1.4pc/[ll]\\
		&&Z_k&Z_\infty\ar_{\nu_k}[l]\ar@/^2.5pc/@{-->}[ull]
	}
\end{equation}
hence $f_\infty\colon X_\infty \to Y$ is the Iitaka fibration.
\end{proof}
Before stating our last theorem we point out the reader that it is only conjectural since in order it to be true we need to assume the following:
\begin{enumerate}
\item let $X$ be a nonsingular projective variety over $Z$. If $K_X$ is pseudoeffective$/Z$ then $X/Z$ has a minimal model. Otherwise it has a Mori fiber space$/ Z$. 
\item Let $U/Z$ be a normal projective variety with terminal singularities and $\mathbb Q$-Cartier canonical divisor. If $K_U$ is nef$/Z$ then it is semiample, that is it is the pull-back of a divisor that is ample$/Z$
\end{enumerate}
The two assumptions put together is referred as the good minimal model conjecture; c.f.see \cite{T}.

\begin{thm}\label{Iitakadue}
Let $X$ be a smooth variety of dimension $n$. Let $\{\omega_1,\dots, \omega_m\} \subset H^0 (X,\Omega^p_X)$ be a $p$-strict subset and let $\sL\hookrightarrow \Omega_X^p$ be the sub-line bundle generically generated by them. Assume that:
\begin{enumerate}
\item $\bigoplus H^0(X, sL)$ is a finitely generated $\mathbb C$-algebra, where $\sL=\sO_X(L)$;
\item $0\leq\Kod X=\Kod \sL=p<n$;
\item $aK_X-bL$ is nef where $a,b\in\mathbb N_{>0}$.
\end{enumerate}
If the good minimal model conjecture holds for terminal projective varieties with zero Kodaira dimension up to dimension $n-p$ then the Iitaka fibration of $\sL$ is the Iitaka fibration.
\end{thm}
\begin{proof} Let $Y:={\rm{Proj}} \bigoplus_{n\in\mathbb N}H^0(X,nL)$.
By \cite[Chapter III Section 1.2]{B} we can assume that there exists $l\in\mathbb N$ such that $\bigoplus_{n\in\mathbb N}H^0(X,nL)$ is generated by $H^0(X,lL)$ and so the natural map $\phi\colon X\dashrightarrow Y$ is induced by $\phi_{\mid lL\mid}\colon X\dashrightarrow\mathbb P(H^0(X,lL)^\vee)$. Since $aK_X-bL$ is nef then $l\cdot(aK_X-bL)$ is nef. Hence, up to consider $l\cdot bL$ instead of $L$, from now on, we can assume that there exists $r\in\mathbb N$ such that $rK_X-L$ is nef and that $H^0(X,L)$ generates $\bigoplus_{n\in\mathbb N}H^0(X,nL)$. 
By unicity up to birational modifications of the Iitaka fibration we obtain that the rational map $f=\colon X\dashrightarrow Y\hookrightarrow \mathbb P(H^0(X,lL)^\vee)$ induced by $\phi_{\mid lL\mid}$ is birationally equivalent to a fixed algebraic fiber space $\phi_\infty\colon X_\infty\to Y_\infty$ where $X_\infty$ is smooth, $Y_\infty$ is normal and as above we obtain 
\begin{equation}
	\xymatrix{
		&&X\ar_{f}@{-->}[d]&X_\infty\ar_{\phi_\infty}[d]\ar^{u_\infty}[l]\\
		&&Y&Y_\infty\ar_{\nu_k}[l]
	}
\end{equation}

\noindent where $\phi_\infty\colon X_\infty \to Y_\infty$ is the Iitaka fibration for $L$. By construction $u_\infty^{*}(rK_X-L)$ is nef. We set $L_\infty:=u_\infty^{*}L$ and also we set:
$$
\mid L_\infty\mid=\mid M^L_\infty\mid + F_L
$$
where $F_L$ is the fixed part of $\mid L_\infty\mid$, $M^L_\infty$ is b.p.f., and it is the pull-back of a divisor from $Y_\infty$. Since the Iitaka fibration is up to birational equivalence we can also set
$$
\mid u_\infty^{*}(rK_X)\mid=\mid M\mid +F_{K_X}
$$
where $\mid M\mid$ is b.p.f. and $F_{K_X}$ is the fixed part of $\mid u_\infty^{*}(rK_X)\mid$.

Up to repeating the first step for a sufficiently high multiple $l\cdot(u_\infty^{*}(rK_X)-L_\infty)$ we can assume that $u_\infty^{*}(rK_X)-L_\infty$ is nef and that the morphism $\phi_{\mid M\mid}\colon X_\infty\to \mathbb P(H^0(X,M)^\vee)$ is factorised through a normal variety $Z_r\hookrightarrow \mathbb P(H^0(X,M)^\vee)$ via a morphism $\psi_r\colon X_\infty\to Z_r$ which has connected fiber of dimension $n-p$ and such that there exists the following commutative diagram

\begin{equation}
	\xymatrix{
		&Y_\infty&X_\infty\ar_{\psi_r}@{-->}[d]\ar^{\phi_\infty}[l]&X^1_\infty\ar_{\psi^1_\infty}[d]\ar^{v_\infty}[l]\ar_{f_\infty}@/_1.4pc/[ll]\\
		&&Z_r&Z_\infty\ar_{\nu_r}[l]\ar@/^2.5pc/@{-->}[ull]
	}
\end{equation}
where $\psi^1_\infty\colon X^1_\infty \to Z_\infty$ is the Iitaka fibration of $u_\infty^*(K_X)$, or, which is the same, that it is also the Iitaka fibration of $K_{X^1_\infty}$. In particular we can also assume $X^1_\infty $ and $Z_\infty$ to be smooth. We claim that $v_\infty^*(u_\infty^{*}(rK_X)-M^L_\infty))$ is nef on the general movable curve $C_\infty$ contained inside the general fiber of $\psi^1_\infty\colon X^1_\infty \to Z_\infty$. Indeed if
$$
v_\infty^*(u_\infty^{*}(rK_X)-M^L_\infty))\cdot C_\infty<0
$$
then letting $C:=v_{\infty_{*}}C_\infty$ it holds that
$$
(u_\infty^{*}(rK_X)-M^L_\infty))\cdot C<0
$$
On the other hand
$$
(u_\infty^{*}(rK_X)-L_\infty))\cdot C\geq 0
$$
then
$$
-F_L\cdot C>0
$$
This means that $F_L\cdot C<0$, that is $C$ is inside the support of $F_L$: a contradiction.

By the same proof it also holds that:
\begin{equation}\label{stessaprova}
	(rK_{X^1_\infty}-v_\infty^* M^L_\infty)\cdot C_\infty\geq 0
\end{equation}

By hypothesis the fibers of $\psi^1_\infty$ have dimension $n-p$ and by construction of the Iitaka fibration the general one has Kodaira dimension $0$.
Since we have assumed that the good minimal model conjecture holds for terminal projective varieties with zero Kodaira dimension up to dimension $n-p$, we can run a MMP over $Z_\infty$ in order to find a terminal variety $X^2$ and a birational map $\tau^{-1}\colon X^1_\infty\dashrightarrow X^2$ such that
\begin{equation}
	\xymatrix{
		&&X^1_\infty\ar_{\psi^1_\infty}[d]&X^2\ar_{\psi^2}[d]\ar^{\tau}@{-->}[l]\\
		&&Z_\infty&Z_\infty\ar_{=}[l]
	}
\end{equation}
and $K_{X^2}$ is $\psi^2$-nef. Moreover due to the fact that $\psi_\infty\colon X^1_\infty \to Z_\infty$ is up to birational equivalence we can assume that $\tau^{-1}\colon X^1_\infty\dashrightarrow X^2$ is actually a birational morphism. 

We stress that $\psi^2\colon X^2\to Z_\infty$ is a surjective morphism with connected fibers between normal projective varieties. Moreover $K_{X^2}$ is a $\psi^2$-nef $\mathbb Q$-Cartier divisor. Since we have assumed that the abundance conjecture holds for varieties of vanishing Kodaira dimension up to dimension $n-p$ we find that the canonical of the general fiber is torsion. By \cite[Lemma 3.4]{T} there exist birational morphisms
$\rho\colon X^2_\infty\to X^2$, $\pi\colon Z^2_\infty\to Z_\infty$, a $\mathbb Q$-Cartier divisor $G$ in $Z^2_\infty$ and an equidimensional morphism $\psi^2_\infty \colon X^2_\infty\to Z^2_\infty$ such that $\rho^*(K_{X^2})=(\psi^2_\infty)^*(G)$. We build the following commutative diagram:

\begin{equation}\label{grandediagramma}
	\xymatrix{
		Y&X\ar@{-->}[l]&&X^1_\infty\ar@{=}[ld]\ar^{\tau^{-1}}[d]&X^3_\infty\ar^{\mu}[d]\ar_{\delta}[l]\\
		Y_\infty\ar[u]&X_\infty\ar_-{\psi_r}@{-->}[d]\ar^{u_\infty}[u]\ar^{\phi_\infty}[l]&X^1_\infty\ar_-{\nu_\infty}[l]\ar^{\psi^1_\infty}[d]&X^2\ar_-{\tau}@{-->}[l]\ar^-{\psi^2}[dl]&X^2_\infty\ar^{\rho}[l]\ar^{\psi^2_\infty}[d]\\
		&Z^1&Z_\infty\ar[l]&&Z^2_\infty\ar^{\pi}[ll]
	}
\end{equation}
where we have resolved $\phi_\infty \circ v_\infty \circ\tau\circ\rho\colon X^2_\infty \dashrightarrow Y_\infty$ and again by the fact that Iitaka construction is up to a birational maps we can assume that there exist birational morphisms $\delta\colon X^3_\infty\to X^1_\infty$, $\mu\colon X^3_\infty\to X^2_\infty$ such that

\begin{equation}\label{prestocommuta}
	\xymatrix{
		&&X^1_\infty\ar_{\tau^{-1}}[d]&X^3_\infty\ar_{\mu}[d]\ar^{\delta}[l]\\
		&&X^2&X^2_\infty\ar_{\rho}[l]
	}
\end{equation}
is commutative and $f_\infty=\phi_\infty\circ\nu_\infty\circ\delta\colon X^3_\infty\to Y_\infty$ is a morphism with connected fiber.

Now a generic movable curve $C_\infty$ inside the general fiber of $\psi^1_\infty\colon X^1_\infty \to Z_\infty$ is transformed in a generic movable curve $C^2_\infty$ inside the general fiber of $\psi^2_\infty\colon X^2_\infty \to Z^2_\infty$.

We set:
\begin{equation}\label{semipresto}C_\infty=(\tau^{-1})^{*}C^{(2)}
\end{equation}
where $C^{(2)}$ is a generic movable curve inside the general fiber of $\psi^2\colon X^2 \to Z_\infty$. By construction $K_{X^2}=\tau^{-1}_{*}K_{X^1_\infty}$ since $X^2$ is terminal and $\tau^{-1}$ is a morphism. Hence:
\begin{equation}\label{presto}
	(rK_{X^2}-(\tau^{-1})_{*}v_\infty^* M^L_\infty )\cdot C^{(2)}=(\tau^{-1}_{*}(rK_{X^1_\infty}-v_\infty^* M^L_\infty)) \cdot C^{(2)}
\end{equation}
By projection formula applied to the morphism $\tau^{-1}$ it holds:
By substitution using equation \ref{semipresto}
\begin{equation}\label{prestopresto}
	(\tau^{-1}_{*}(rK_{X^1_\infty}-v_\infty^* M^L_\infty)) \cdot C^{(2)}=(rK_{X^1_\infty}-v_\infty^* M^L_\infty)\cdot (\tau^{-1})^{*}C^{(2)}
\end{equation}
By substitution using equation \ref{semipresto}

\begin{equation}\label{prestoprestopresto}
	(rK_{X^1_\infty}-v_\infty^* M^L_\infty)\cdot (\tau^{-1})^{*}C^{(2)}=(rK_{X^1_\infty}-v_\infty^* M^L_\infty)\cdot C_\infty
\end{equation}

Since by equation \ref{stessaprova} we know that $(rK_{X^1_\infty}-v_\infty^* M^L_\infty)\cdot C_\infty\geq 0$ we see by equations \ref{presto}, \ref{semipresto} and \ref{prestopresto} that

\begin{equation}\label {nontardi}
	(rK_{X^2}-(\tau^{-1})_{*}v_\infty^* M^L_\infty )\cdot C^{(2)}\geq 0
\end{equation}
Finally let $C^{(2)}_{\infty}$ be a generic movable curve inside the general fiber of $\psi^2_\infty\colon X^2_\infty \to Z^2_\infty$.
It holds that
$$
\rho^{*}(rK_{X^2}-(\tau^{-1})_{*}v_\infty^* M^L_\infty)C^{(2)}_{\infty}\geq 0
$$
On the other hand since $\rho^{*}(K_{X^2})=(\psi^2_\infty)^*(G)$ it holds that 
$$\rho^{*}(rK_{X^2})\cdot C^{(2)}_{\infty}=0.$$ Hence 
\begin{equation}\label{eccoilpresto}
	-\rho^{*}((\tau^{-1})_{*}v_\infty^* M^L_\infty)C^{(2)}_{\infty}\geq 0
\end{equation}
Since $M^L_\infty$ is movable by equation \ref{eccoilpresto} it holds that
\begin{equation}\label{eccoilprestopresto}
	\rho^{*}((\tau^{-1})_{*}v_\infty^* M^L_\infty)C^{(2)}_{\infty}=0
\end{equation}

By the commutative diagram \ref{prestocommuta} we obtain:
\begin{equation}\label{finepresto}
	\delta^* (v_\infty^* M^L_\infty)\cdot \mu^{*} C^{(2)}_{\infty}=0
\end{equation}
The equation \ref{finepresto} shows that the general fiber of $\psi^2_\infty\circ\mu\colon X^3_\infty\to Z^2_\infty$ is the general fiber of $f_\infty\colon X^3_\infty\to Y_\infty$, by the rigidity lemma we conclude. 
\end{proof}

\end{document}